\documentclass[onefignum,onetabnum]{siamart171218}

\usepackage{braket,amsfonts}

\usepackage{array}

\usepackage[caption=false]{subfig}
\usepackage{pgfplots}

\newsiamthm{claim}{Claim}
\newsiamremark{remark}{Remark}
\newsiamremark{hypothesis}{Hypothesis}
\crefname{hypothesis}{Hypothesis}{Hypotheses}
\newsiamremark{expl}{Example}

\usepackage{algorithmic}
\usepackage{enumitem}\setlist[enumerate]{leftmargin=.5in}\setlist[itemize]{leftmargin=.5in}
\usepackage{graphicx,epstopdf}
\usepackage{amssymb}
\Crefname{ALC@unique}{Line}{Lines}

\usepackage{amsopn}

\usepackage{xspace}
\usepackage{bold-extra}
\usepackage[most]{tcolorbox}

\colorlet{texcscolor}{blue!50!black}
\colorlet{texemcolor}{red!70!black}
\colorlet{texpreamble}{red!70!black}
\colorlet{codebackground}{black!25!white!25}

\lstdefinestyle{siamlatex}{%
  style=tcblatex,
  texcsstyle=*\color{texcscolor},
  texcsstyle=[2]\color{texemcolor},
  keywordstyle=[2]\color{texemcolor},
  moretexcs={cref,Cref,maketitle,mathcal,text,headers,email,url},
}

\tcbset{%
  colframe=black!75!white!75,
  coltitle=white,
  colback=codebackground, 
  colbacklower=white, 
  fonttitle=\bfseries,
  arc=0pt,outer arc=0pt,
  top=1pt,bottom=1pt,left=1mm,right=1mm,middle=1mm,boxsep=1mm,
  leftrule=0.3mm,rightrule=0.3mm,toprule=0.3mm,bottomrule=0.3mm,
  listing options={style=siamlatex}
}

\newtcblisting[use counter=example]{example}[2][]{%
  title={Example~\thetcbcounter: #2},#1}

\newtcbinputlisting[use counter=example]{\examplefile}[3][]{%
  title={Example~\thetcbcounter: #2},listing file={#3},#1}

\DeclareTotalTCBox{\code}{ v O{} }
{ 
  fontupper=\ttfamily\color{black},
  nobeforeafter,
  tcbox raise base,
  colback=codebackground,colframe=white,
  top=0pt,bottom=0pt,left=0mm,right=0mm,
  leftrule=0pt,rightrule=0pt,toprule=0mm,bottomrule=0mm,
  boxsep=0.5mm,
  #2}{#1}

\begin{tcbverbatimwrite}{tmp_\jobname_header.tex}
\title{\MakeLowercase{$hp$}-Multilevel Monte Carlo Methods for Uncertainty Quantification of Compressible Navier-Stokes Equations\thanks{Submitted to the editors \today.
\funding{The authors thank the Baden-W\"urttemberg Stiftung for financial support via the project ``BW-HPC2: SEAL'' and the High-Performance Computing Center Stuttgart (HLRS) 
for the provided computing resources.}}}

\author{Andrea Beck\thanks{Institute of Aerodynamics and Gas Dynamics, University of Stuttgart, Pfaffenwaldring 21, 70569
Stuttgart, Germany.}
\and Jakob D\"{u}rrw\"{a}chter$^\dagger$
\and Thomas Kuhn$^\dagger$
\and Fabian Meyer\thanks{Institute of Applied Analysis and Numerical Simulation, University of Stuttgart, Pfaffenwaldring 57,
70569 Stuttgart, Germany. (\email{fabian.meyer@mathematik.uni-stuttgart.de)}}
\and Claus-Dieter Munz$^\dagger$
\and Christian Rohde$^\ddagger$
}

\headers{}{}
\end{tcbverbatimwrite}
\input{tmp_\jobname_header.tex}
\input{macros.sty}

\ifpdf
\hypersetup{ pdftitle={$hp$-MLMC} }
\fi

\begin{document}
\maketitle

\begin{tcbverbatimwrite}{tmp_\jobname_abstract.tex}
\begin{abstract} 
  We propose a novel $hp$-multilevel Monte Carlo method for the
  quantification of uncertainties in the compressible Navier-Stokes
  equations, using the Discontinuous Galerkin method
  as deterministic solver. The  multilevel approach exploits hierarchies 
  of uniformly refined meshes while simultaneously
  increasing the polynomial degree of the ansatz space. It allows for 
  a very large range of resolutions in the physical space and thus 
  an efficient decrease
  of the statistical error.
  We prove that the overall complexity of the $hp$-multilevel Monte Carlo method
  to compute the mean field
  with prescribed accuracy is, in best-case,  
  of quadratic order with respect to the accuracy.
  We also propose
  a novel and simple approach to estimate a lower confidence bound for the optimal number of
  samples per level, which helps to prevent
  overestimating these quantities.
  The method is in particular designed for application
  on queue-based computing systems, where
  it is desirable to compute a large number of samples during one
  iteration, without overestimating the optimal number of samples.
  Our theoretical results are verified by numerical experiments
  for the two-dimensional compressible Navier-Stokes equations.
  In particular we consider a cavity flow problem from
  computational acoustics, demonstrating
  that the method is suitable to handle complex engineering problems.
\end{abstract}
\end{tcbverbatimwrite}
\input{tmp_\jobname_abstract.tex}

%
\section{Introduction}
Due to the continuous improvement of computer-processing capacities, the demand for highly accurate numerical simulations
which also account for uncertain input parameters is growing.
Uncertainties might arise from limitations in measuring physical phenomena exactly
or from a systematical absence of knowledge about the underlying physical processes.
Uncertainty Quantification (UQ) addresses this issue and provides mathematical methods to quantify the influence of uncertain
input parameters on the numerical solution itself or on derived quantities of interest.
There exist two major approaches for UQ. On the one hand, non-statistical approaches
like the intrusive and non-intrusive polynomial chaos expansion approximate the underlying random field
by a series of polynomials and derive deterministic models for the stochastic modes.
On the other hand, statistical approaches such as Monte Carlo (MC) type methods sample the random space to obtain statistical
information,
like mean, variance or higher order moments of the corresponding random field. Especially MC
type methods are very popular as they are easy to implement and only require a deterministic black box solver.
Moreover, in contrast to non-statistical approaches, the MC method does not rely on 
the regularity of the underlying random
field, rendering it a very robust method.
However, the convergence of MC methods is dictated by the law of large numbers, hence relatively slow and therefore
computationally expensive.

To overcome these difficulties Heinrich \cite{Heinrich2001} and later Giles \cite{Giles2008}
extended the MC method to the Multilevel Monte Carlo (MLMC) method, where they considered different mesh hierarchies instead
of one fixed mesh,
to discretize the deterministic equation of interest.
The MLMC method relies on the idea that the global behavior of the exact expectation can be approximated by the behavior of
the expectation of numerical solutions with a low spatial resolution,
 which can be computed at low cost.
The coarse expectation is then subsequently corrected by computations on finer meshes, which are computationally more expensive
per sample.
The number of these simulations at full resolution is significantly reduced compared to the original MC method
resulting in a considerably lower overall computational cost.
Since its development the MLMC method has been very successfully applied to UQ for many different
partial differential equations  with uncertainties, as for example in
\cite{BarthMLMCHyperbolic,BarthMLMCParabolic,BarthMLMC,CliffeScheichl2011,NobileCMLMC,
HoelKrumscheid2019,MishraSchwabSukysNonlinearSystems,AppeloeMOMC}.
We refer to \cite{ChenMing2018,Nobile2017} for applications of the MLMC method 
in computational
fluid dynamics, especially to aerodynamics and meteorology. 
A generalization of
the MLMC method, named Multiindex Monte Carlo (MIMC) method, has been presented in \cite{TemponeNobile2016}. 
In contrast to the MLMC method, which computes mean and variance 
using first-order differences, the
MIMC  method uses high-order differences which allows for 
a faster decay of the corresponding level variances, resulting in significant efficiency gains
compared to the standard MLMC method.

In \cite{AppeloeMOMC} the authors extended the MLMC method for hyperbolic problems 
to a Multiorder Monte Carlo method (MOMC),
using an energy-preserving Discontinuous Galerkin (DG) scheme
for the elastic wave equation, which we will dub 
$p$-MLMC for the remainder of this article.
The authors considered either mesh refinements ($h$-refinement), 
or increased the DG polynomial degree ($p$-refinement)
to obtain a hierarchy of different levels.
Furthermore, they proved  that the computational complexity to reach
a prescribed accuracy is of quadratic order with respect to the accuracy.
However, a proof for the computational complexity 
for a hierarchy of $hp$-refined meshes is still open.
We therefore extend the $h$- and $p$-MLMC method to an $hp$-MLMC method, 
where we uniformly refine the physical mesh and  at the same time uniformly increase the DG
polynomial degree. 
As the first new contribution of this work,
we extend the complexity analysis from \cite{CliffeScheichl2011, AppeloeMOMC} 
to  arbitrarily $hp$-refined meshes and show
that the $hp$-MLMC method is - up to a constant - as efficient as
the  $h$-MLMC and the $p$-MLMC method.
The $hp$ mesh hierarchy enables us to cover a very large range of resolution levels, 
which is crucial for the efficiency of the $hp$-MLMC method. 
Furthermore, from a numerical point of view, a low polynomial degree 
(resulting in a more dissipative numerical scheme) might be favorable 
in connection with coarse meshes, where an insufficient resolution
can otherwise lead to unphysical oscillations, whereas a high 
polynomial degree yields higher accuracy when fine meshes are employed.
The $hp$-MLMC method becomes an attractive alternative for mesh-based MLMC methods if 
 uniform mesh  hierarchies are not available. This might occur for complex domain geometries or 
	when using dynamical mesh adaption. In view of the recent rise of higher-order schemes it seems likely 
	to exploit a mixture method. 
In contrast to $h$-refinement, $hp$-refinement 
introduces more degrees of freedom 
in designing a suitable mesh hierarchy.

For complex fluid dynamical problems, e.g. direct numerical simulations of
unsteady, compressible Navier-Stokes equations, it is inevitable 
to use large-scale computing systems.
These systems are in most cases equipped with a queuing system. 
When using the $hp$-MLMC method on queue-based, large-scale computing systems, 
due to long queuing times, it is desirable to compute as many samples as possible
per iteration. On the other hand, to avoid unnecessary computations and 
to increase the efficiency of the $hp$-MLMC method, one is interested to not overestimate
the optimal number of samples per iteration.
Therefore, as the second novel 
contribution, we show how to construct easily a robust lower confidence bound for the number of
samples per level, which avoids overshooting the optimal number of samples,
but still yields a reasonably large number of additional samples per iteration.
To demonstrate the efficiency of the $hp$-MLMC method combined with the novel sample estimator 
we apply our method to
two different compressible flow problems, a benchmark problem with smooth solution
and an open cavity flow problem. The latter is an important
problem in computational acoustics that exhibits physical phenomena with high sensitivity
with respect to the problem parameters \cite{Kuhn2018}. 
Moreover, we provide for both problems a thorough  comparison of the  $h$-, $p$- and $hp$-MLMC methods.

This article is structured as follows.
In \secref{sec:prelim} we introduce the necessary mathematical framework and briefly introduce the DG method.
In \secref{sec:MLMC} we describe the $hp$-MLMC method and prove the stated complexity result.
We also discuss the necessity of confidence intervals for the estimate of the optimal number of samples
when working on queue-based large-scale computing systems.
Finally, in \secref{sec:numerics} we apply our method to two different examples
and verify our theoretical results.
%
\section{Notation and Preliminaries} \label{sec:prelim}
\subsection{A Primer on Probability Theory}
We let $(\Omega, \F,\P)$ be a probability space, where $\Omega$ is the set of all elementary events $\omega \in \Omega$, $\F$
is a $\sigma$-algebra
on $\Omega$ and $\P$ is a probability measure. We further consider a second measurable space ($E, \mathcal{B}(E))$, where $E$
is a Banach space
and $\mathcal{B}(E)$ is the corresponding Borel $\sigma$-algebra. An $E$-valued random field  is any mapping $X:\Omega \to E$
such that
$\{\omega \in \Omega: X(\omega) \in B\}\in \F$ holds for any $B\in \mathcal{B}(E)$.
For $ r \in [1,\infty) \cup \{\infty\}$ we consider the Bochner space $\leb{r}{\Omega}{E}$ of $r$-summable $E$-valued random
variables $X$ equipped with the norm
\begin{align*}
  \|X\|_{\leb{r}{\Omega}{E}} :=
  \begin{cases}
    (\int_\Omega \|X(\omega)\|_E^r ~\dPOmega)^{1/r} = \E{\|X\|_E^r}^{1/r},\quad& 1\leq r<\infty \\
    \operatorname{ess~sup}\limits_{\omega \in \Omega} \|X(\omega)\|_E,	&r= \infty .
  \end{cases}
\end{align*} 
The uncertainty is introduced via a random vector
$\xi(\omega) = \big(\xi_1(\omega),\ldots, \xi_N(\omega)\big)  : \Omega\to  \Xi \subset \R^N$ with independent, absolutely continuous random variables as components.
This means that for each random variable $\xi_i$ there exists a density function
$f_{\xi_i}:\R \to \R_+$, such that $\int_{\R} f_{\xi_i}(y)~ \mathrm{d}y=1$ and	$\P[ \xi_i \in A]=\int_A
f_{\xi_i}(y)~ \mathrm{d}y$
for any $A \in \mathcal{B}(\R)$ and for all $i=1,\ldots, N$.
Moreover, the joint density function $f_{\xi}$ of the random vector
 $\xi = (\xi_1,\ldots,\xi_N)$ can be written as
$ f_{\xi}(y) = \prod \limits_{i=1}^N f_{\xi_i}(y_i)\quad \text{for all } ~ y= (y_1,\ldots, y_N)^\top \in \Xi.$
The random vector induces a probability measure $\tilde{\P}(B):= \P( \xi^{-1}(B))$ for all $B \in \mathcal{B}(\Xi)$ on the
measurable space $(\Xi, \mathcal{B}(\Xi))$. This measure is called the law of $\xi$ and in the following we work on the
image probability space $(\Xi, \mathcal{B}(\Xi), \tilde{\P})$.

\subsection{The Random Navier-Stokes Equations}
As the physical space we consider a bounded domain $D\subset \R^2$ and we further define the
space-time-stochastic domain $D_{T,\Xi}:= (0,T) \times D\times \Xi$.
We focus on the random compressible Navier-Stokes equations given by
\begin{equation} \label{eq:NSE}
  U_t+{\nabla}_x \cdot ({G}(U)-{H}(U,{\nabla}_xU))=0, \qquad  \forall ~ (t,{x},y) \in D_{T,\Xi},
\end{equation}
where $U(t,{x},y)$ denotes the solution vector of the conserved quantities, i.e. we have $U=(\rho,\rho v_1, \rho v_2,\rho e)^\top$.
  ${G}$ and ${H}$ denote the advective  and viscous fluxes, 
i.e.
\begin{equation}
  {G}_{i}(U)=\begin{pmatrix}\rho\, v_i\\ \rho\, v_1 v_i +\delta_{1i}\,p\\
  \rho\, v_2 v_i +\delta_{2i}\,p\\ \rho\,ev_i + p\,v_i\end{pmatrix},\; \quad
  {H}_{i}(U,{\nabla}_{x}U)=\begin{pmatrix}0\\ \tau_{1i}\\
  \tau_{2i} \\ \tau_{ij}v_j - q_i\end{pmatrix}, \quad i=1,2.
\end{equation}
Here, $\delta_{ij}$ is the Kronecker delta function and the physical quantities $\rho$, ${v}=(v_1,v_2)^\top$, $p$, and $e$
represent density, the velocity vector, the pressure and the specific total energy, respectively. With Stokes' and Fourier's hypothesis,
the viscous stress tensor ${\tau}$ and the heat flux ${q}$ reduce to
\begin{equation} \label{def:tau}
  {\tau}=\mu ({\nabla} {v} + ({\nabla}{v})^\top - \frac{2}{3}({\nabla}\cdot {v}){I}), \quad
  {q}=-k{\nabla}\mathcal{T}.
\end{equation}
In \eqref{def:tau} $\mu$ is the dynamic viscosity, 
$k$ the thermal conductivity and $\mathcal{T}$ the local temperature. In order to
solve for the unknowns, the system has to be closed by appropriate equations of state. 
We choose for the gas constant $R$, the adiabatic exponent $\kappa$ and the specific 
heat at constant volume $c_v$ the perfect gas law assumptions
\begin{equation}
  p=\rho R \mathcal{T}=(\kappa-1)\rho(e-\frac 1 2 {v}\cdot{v}),\qquad e=\frac{1}{2}{v}\cdot{v} + c_v \mathcal{T}.
\end{equation}
We augment \eqref{eq:NSE} with suitable boundary and initial conditions, denoted by
\begin{align*}
B(U) = g  \quad \forall~ (x,y) \in \partial D \times \Xi, \quad U(0,x,y)=U^0(x,y)  \quad \forall ~(x,y) \in D \times \Xi. 
\end{align*}
The boundary operator $B$, the boundary data $g$ and the initial condition $U^0$ 
will be specified when we detail the settings for
the numerical experiments in \secref{sec:numerics}.

Following \cite{SchwabMishraMLMC} we call $U\in L^2(\Xi;C^1([0,T];L^2(D)))$ a weak random solution of \eqref{eq:NSE},
if it is a weak solution
 $\tilde{\P}$-a.s. $y \in \Xi$
and a measurable mapping $\Big(\Xi, \mathcal{B}(\Xi)\Big) \ni y \to U(\cdot,\cdot,y) 
\in  \Big(C^1([0,T];L^2(D)),
\mathcal{B}\big(C^1([0,T];L^2(D))\big)\Big)$.
%
\subsection{The Runge--Kutta Discontinuous Galerkin Method}\label{sec:DG}
We shortly recall the Discontinuous Galerkin (DG) spatial discretization for the
initial-boundary value problem \eqref{eq:NSE}, see \cite{HindenlangGassner2012}
for more details. To partition the 
spatial domain we subdivide $D$ into $\nSpace\in \N$ quadrilateral elements $D_m$, $m=1,\ldots, \nSpace$ with $D= \bigcup \limits_{m=1}^{\nSpace} D_m$
and define the mesh size  $h:= \max \limits_{m=1,\ldots,\nSpace} h_m$, where $h_m$ is the length of~$D_m$. Moreover, we define $h_{min}:= \min \limits_{m=1,\ldots,\nSpace} h_m$.
Furthermore, let us introduce the space of piecewise DG polynomial ansatz and test functions:
$
  \mathcal{V}_{h}^q := \{ U : D \to \R^{4}~|~ U\big|_{D_m} \in \P_q(D_m; \R^{4}),~\text{for } 1\leq m \leq {\nSpace}\},
$
where $\P_q(D_m; \R^{4})$ is the space of polynomials of degree $q$ on the element $D_m$ 
and $U\big|_{D_m}$
denotes $U$ restricted to $D_m$. The DG solution $U_h$ is then sought in  $\mathcal{V}_{h}^q$, i.e.
$U_h(t,\cdot,y) \in  \mathcal{V}_{h}^q$, a.e. $(t,y) \in (0,T)\times \Xi$.
On each element $D_m$, $m=1,\ldots, \nSpace$ we use tensor products of local one-dimensional Lagrange interpolation
polynomials of degree $q$, i.e.
\begin{align}
  U_h(t,x, y) \big|_{D_m} = \sum \limits_{i,j=0}^q U_{i,j}^m(t,y) l_i^m (x_1) l_j^m(x_2).
\end{align}
The interpolation nodes are chosen to be the Gau{\ss}-Legendre nodes,
cf. \cite{HindenlangGassner2012}.
We then consider the (spatial) weak form of \eqref{eq:NSE} given by  
\begin{align}\label{eq:dg}
  \frac{\partial}{\partial t}\int_{D}  & U(t,\cdot,y)  \Phi\, ~\mathrm{d}{x}   +\oint_{\partial D}
  (\mathcal{G}(U(t,\cdot,y))-\mathcal{H}(U(t,\cdot,y),{\nabla}_xU(t,\cdot,y)))\Phi\,~\mathrm{d}s  \notag \\
   & - \int_{D} ({G}(U(t,\cdot,y))-{H}(U(t,\cdot,y),{\nabla}_x U(t,\cdot,y)))\cdot{\nabla}_{x}\Phi\,~ \mathrm{d}{x} =0
\end{align}
for a.e. $t \in (0,T)$ and $\tilde{\P}$-a.s. $y \in \Xi$ and for all test functions $\Phi$.
Now, using the same discrete space $\mathcal{V}_{h}^q$ for ansatz and test functions
in \eqref{eq:dg},
we obtain  the    semi-discrete  DG scheme for $U_h \in L^2(\Xi; C^1([0,T);\mathcal{V}_{h}^q ))$:
\begin{align}\label{eq:semiDiscrete}
  &  \frac{\partial}{\partial t}\int_{D}  U_h(t,\cdot,y)  \Phi_h\,~ \mathrm{d}{x}   +\oint_{\partial D}
  \mathcal{G}_{n}^*(U_h^-(t,\cdot,y),U_h^+(t,\cdot,y))\Phi_h\,~\mathrm{d}s  \notag \\ &  +\oint_{\partial
  D} \mathcal{H}^*_n(U_h(t,\cdot,y),{\nabla}_xU_h(t,\cdot,y))\Phi_h\,~\mathrm{d}s - \int_{D}
  {G}(U_h(t,\cdot,y))\cdot{\nabla}_{x}\Phi_h\,~ \mathrm{d}{x} \notag \\
   &+ \int_{D}	{H}(U_h(t,\cdot,y),{\nabla}_x U_h(t,\cdot,y))\cdot{\nabla}_{x}\Phi_h\,~ \mathrm{d}{x} =0
\end{align}
for all $\Phi_h \in \mathcal{V}_{h}^q $.
Here, $\mathcal{G}_{n}^*(U^-,U^+)$ denotes a numerical flux,
which depends on values at the grid cell interface from neighboring 
cells.
In this paper, we have chosen the approximate Roe Riemann solver with entropy fix described in 
\cite{Harten1983}.
The viscous fluxes $\mathcal{H}^*_n$ normal to the cell interfaces are approximated by the procedure described by Bassi and
Rebay in \cite{BassiRebay1997}. 
The DG scheme \eqref{eq:semiDiscrete} is then advanced in time by a $(q+1)$-th order Runge--Kutta method
\cite{KennedyCarpenterRK} constrained by a CFL type condition of the form 
\begin{align} \label{def:cfl}
\Delta t \leq \min \Big\{ \frac{h_{min}}{\lambda_{max}^c (2q+1)}, 
\Big(\frac{h_{min}}{\lambda_{max}^v (2q+1)}\Big)^2 \Big\}.
\end{align}
In \eqref{def:cfl} $\lambda_{max}^c:=\big((|v_1|+c)+(|v_2|+c) \big)$ is 
an estimate for the absolute value of the largest eigenvalue of the
convective flux Jacobian 
with $c:=\sqrt{\kappa \frac{p}{\rho}}$ being the speed of sound.
Moreover, 
$\lambda_{max}^v:=\Big(\max\Big(\frac{4}{3 \rho},\frac{\kappa}{p}\Big)
\frac{\mu}{Pr}\Big)$ is an estimate
for the largest eigenvalue of the diffusion matrix of the viscous flux,
$Pr= \frac{c_p \mu}{k}$ being the Prandtl number.
With this choice the consistency error of the numerical scheme is 
formally of order $\mathcal{O}(h^{q+1} +\Delta t^{q+1})$.
We note that we indicate the numerical solution  in the remaining part of this paper
by the spatial parameter $h$ only.

%
\section{The $hp$-Multilevel Monte Carlo Method} \label{sec:MLMC}

\subsection{Description of the $hp$-MLMC Method}
In this section we introduce the $hp$-MLMC method, based on the classical MLMC method from
\cite{Giles2008}. 
For  levels $l=0,\ldots,L$, we consider spatial meshes with $N_l\in
\N$ elements and ansatz spaces of polynomial
degree $q_l\in \N$. We choose the number of elements $N_l$ and the DG polynomial
degrees $q_l$, such that $N_0<\cdots<N_L$ and $q_0<\cdots<q_L$ holds, i.e.,
we simultaneously increase the mesh size and the DG polynomial degree.
With $\mathcal{V}_{h_l}^{q_l}$ we denote the DG polynomial
space corresponding to level $l=0,\ldots,L$ .
Moreover, by  
$U_l(t,\cdot,y):= U_{h_l}(t,\cdot,y) \in \mathcal{V}_{h_l}^{q_l} $ (a.e.~$(t,y)\in (0,T)\times
\Xi)$ we denote the DG numerical solution
associated with level $l \in \{0,\ldots,L\}$.
Additionally, the deterministic numerical solution of \eqref{eq:semiDiscrete} 
on level $l=0,\ldots,L$ for input parameter $y_i \in \Xi$
is denoted by $U_l^i:=U_l(\cdot,\cdot,y_i)$ and will be called sample
for the remaining part of this paper.
Since we do not enforce a global CFL time-step restriction 
across different discretization levels, each sample has to obey the CFL condition
\eqref{def:cfl}. Thus, coarser levels admit a bigger time-step than fine levels.
We want to emphasize that, in contrast to the MIMC method from \cite{TemponeNobile2016}
we index mesh-size and polynomial degree by a single parameter.

It is our goal to compute statistical moments like expected value or higher order moments 
of a general Quantity of Interest
(QoI) $Q(U)$ of the random weak solution  $U$ of \eqref{eq:NSE}. Precisely, we are interested to determine
\begin{align}\label{eq:exactExp}
  \E{Q(U(t,x,y))}  = \int_{\Xi} Q(U(t,x,y)) f_{\xi}(y) \d{y}
\end{align}
for a.e. $(t,x) \in (0,T)\times D$. Here $Q$ can be an arbitrary nonlinear 
function or functional of $U$.
To ease notation we suppress the dependence of the QoI on $(t,x,y)$ and write $Q(U)$.
We approximate \eqref{eq:exactExp} with a Monte Carlo estimator.
To this end, we let   $\{U_L^i\}_{i=1}^M$
be $M \in \N$ independent, identically distributed  samples.
The MC estimator for \eqref{eq:exactExp} is then defined by 
\begin{align} \label{eq:numExp}
  \EMC{M}{Q(U)} :=   \frac{1}{M} \sum \limits_{i=1}^{M} Q(U_L^i) \approx \E{Q(U)}.
\end{align}
Next we advance the MC estimator $\EMC{M}{\cdot}$ to the $hp$-MLMC estimator $\EMLMC{L}{\cdot}$ by using the linearity of
the expectation in combination with a telescoping sum.
 We then write (see \cite{Giles2008})
\begin{align}\label{eq:telescopeSum}
  \E{Q(U_L)} = \sum \limits_{l=0}^L \E{Q(U_l)-Q(U_{l-1})},
\end{align}
where we used the definition $Q(U_{-1})=0$. Now, each term in \eqref{eq:telescopeSum} can be estimated by the MC estimator \eqref{eq:numExp}. If we let $M_l\in \N$ 
denote a level-dependent number of samples for each level $l=0,\ldots,L$ and assume that the samples
$\{Q(U_l^i)\}_{i=1}^{M_l}$, $l=0,\ldots,L$, on different levels
are independent from each other, we
obtain the $hp$-MLMC estimator via
\begin{align*}
  \EMLMC{L}{Q(U_L)} &:= \sum \limits_{l=0}^L \frac{1}{M_l} \sum \limits_{i=1}^{M_l}   (Q(U_l^i) - Q(U_{l-1}^i)) 
  =   \sum \limits_{l=0}^L \EMC{M_l}{Q(U_l) - Q(U_{l-1})} \\
  & \approx \sum \limits_{l=0}^L \E{Q(U_l)-Q(U_{l-1})} = \E{Q(U_L)}\nonumber.
\end{align*}
Here, $Q(U_l^i)$ and $Q(U_{l-1}^i)$, are computed using the same sample
$y_i^l \in \Xi$.
The main idea of the MLMC estimator is
that the global behavior of the exact expectation can be approximated by the behavior of the expectation of numerical solutions
with low resolution,
where each sample can be computed with low cost. Thus,
 $M_l$ is supposed to be large for coarse levels.
The coarse-level expectation 
is then successively corrected by a few computations on finer levels.
Each fine-level sample is
computationally expensive and therefore,
$M_l$ is supposed to be small on fine levels.
Hence, the most important aspect for the efficiency of the $hp$-MLMC  estimator is
the correct choice  of $M_l$. In the following section we want to derive the best choice for $M_l$ such that
the total work is minimized under the constraint that the spatial and stochastic error satisfy
a certain threshold.
\subsection{Optimal Number of Samples}
For the following analysis we set the QoI to be the solution
itself at a fixed point $t\in [0,T]$ in time,
i.e. 
\begin{align} \label{eq:qoi}
Q(U)=U(t,\cdot,\cdot).
\end{align}
We note that our analysis can be analogously performed
using any other QoI, including functional ones and suitable norms for $Q$.
\begin{remark}
We consider the QoI \eqref{eq:qoi} because we are
mainly interested in statistical quantities of the solution $U$ itself.
However, many other QoIs have a higher regularity than the solution $U$, 
especially if $Q$ is a functional. Therefore, using  such a QoI 
yields a faster decrease of the variance across different levels, which
increases the performance of the MLMC method compared to MC. 
In \cite[Fig. 10]{abgrallmishra2017}
it has been numerically shown that for an uncertain Kelvin-Helmholtz problem 
the level variance of the solution does not decrease, because upon each mesh 
refinement additional smaller scale structures are detected. 
Thus, the MLMC method provides
no computational gains compared to the MC method when the QoI is  
the solution  itself. 
\end{remark}
With the help of the following representation of the 
root mean square error (RMSE)
we derive an optimal number of samples $M_l$ for all $l=0, \ldots,L$.
\begin{align}
 \text{RMSE}&:= \| \E{U(t,\cdot,\cdot)}- \EMLMC{L}{U_L(t,\cdot,\cdot)}\|_{L^2(\Xi;\Lp{2}{D})}   \leq \epsilon_{\mathrm{det}} +
  \epsilon_{\mathrm{stat}}. \label{eq:splittingMSE} \\
   \epsilon_{\mathrm{det}}&:=\|\E{U(t,\cdot,\cdot)}- \E{U_L(t,\cdot,\cdot)}\|_{\Lp{2}{D}}, \notag \\
   \epsilon_{\mathrm{stat}}&:=\|\E{U_L(t,\cdot,\cdot)}- \EMLMC{L}{U_L(t,\cdot,\cdot)}\|_{L^2(\Xi;\Lp{2}{D})}. \notag
\end{align}
The  term $\epsilon_{\mathrm{det}}$ in \eqref{eq:splittingMSE}
 is the deterministic approximation error (bias).
It accounts for the insufficient resolution of the deterministic system.
The  term  $\epsilon_{\mathrm{stat}}$ corresponds to the statistical (sampling) error. Its occurs due to the finite number
of samples in \eqref{eq:numExp}.
We choose the optimal number of samples $M_l$ to minimize this term.
For notational convenience we suppress the explicit dependence on $t\in [0,T]$.
Using the independence of the samples, we rewrite the statistical error in \eqref{eq:splittingMSE}
(cf. \cite{MuellerJennyMeyerTwoPhase}):
\begin{align}\label{eq:statisticalError}
  \epsilon_{\mathrm{stat}}^2 & =
  \mathbb{E}\Big [\|\E{U_L}- \EMLMC{L}{U_L}\|_{\Lp{2}{D}}^2 \Big]  & \nonumber
  \\
  &= \mathbb{E}\Big [\Big \|\sum \limits_{l=0}^L \E{U_l-U_{l-1}}
   - \EMC{M_l}{U_l-U_{l-1}}\Big \|_{\Lp{2}{D}}^2\Big]& \nonumber
  \\
  & = \sum \limits_{l=0}^L \frac{1}{M_l^2} \sum \limits_{i=1}^{M_l}  \mathbb{E}\Big [\| \E{U_l -U_{l-1}}-(U_l^i - U_{l-1}^i)\|_{\Lp{2}{D}}^2
  \Big] & \nonumber
  \\
  & = \sum \limits_{l=0}^L \frac{1}{M_l} \mathbb{E}\Big [\| \E{U_l -U_{l-1}}-(U_l- U_{l-1})\|_{\Lp{2}{D}}^2 \Big]   & \nonumber
  \\
  & =:	\sum \limits_{l=0}^L \frac{\sigma_l^2}{M_l}. &
\end{align}
\begin{remark}
We use the $\Lp{2}{D}$-norm instead of the $\Lp{1}{D}$-norm for two reasons.
First, the $\Lp{1}{D}$-norm is appropriate  for
inviscid  flow problems in one spatial dimension, 
but we are interested in regular solutions of  \eqref{eq:NSE}.
Second, using the $\Lp{2}{D}$-norm has the side effect
that \eqref{eq:statisticalError} is satisfied as equality (cf. \cite[Sec. 5.2]{MuellerJennyMeyerTwoPhase}). 
\end{remark} 

From this representation, the optimal number of samples can be obtained by an error-complexity analysis as in \cite{Giles2008,
AppeloeMOMC, MuellerJennyMeyerTwoPhase}.
 We introduce the total work
\begin{align}\label{def:computationalTime}
  \Wtot:=\Wtot(M_0,\ldots,M_L) := \sum \limits_{l=0}^L M_l w_l ,
\end{align}
where $w_l$ is the work needed to create one sample $U_l^i-U_{l-1}^i$. Following \cite{Giles2008}
we obtain the optimal number of samples on different levels by considering the following minimization problem.
\begin{align}\label{def:mininmization}
 \text{For a tolerance } \epsilon>0, \underset{M_0,\ldots,M_L \in \N}{\text{minimize }} \Wtot \text{ under the constraint } \sum \limits_{l=0}^L \frac{\sigma_l^2}{M_l}
  \leq \frac{1}{4}\epsilon^2.
\end{align}

The minimization problem can be explicitly solved by (cf. \cite{Giles2008,
AppeloeMOMC})
\begin{align}
  M_l = \bigg\lceil \Big(\frac{\epsilon}{2}\Big)^{-2} \sqrt{\frac{\sigma_l^2}{w_l}}\sum \limits_{k=0}^L \sqrt{\sigma_k^2 w_k}
  \bigg \rceil.
  \label{eq:Ml}
\end{align}
In this work we consider an iterative version of the 
$hp$-MLMC method, which means that we initialize 
the algorithm with a number of warm-up samples 
and then estimate each quantity in \eqref{eq:Ml} 
 during each iteration
of the $hp$-MLMC method. In order to keep track of 
the number of samples 
that have already been calculated during the iterations of the algorithm
 we denote this quantity by  $M_{\mathrm{tot}_l}$ for
 each level $l\in \{0,\ldots,L\}$.
We want to emphasize, that $M_{\mathrm{tot}_l}$ is not equal to $M_l$,
which is the estimated optimal number of samples from \eqref{eq:Ml}.

The level variances $\sigma_0^2,\ldots\sigma_L^2$ in \eqref{eq:statisticalError} 
are not known in general and we therefore estimate the level variance
using the unbiased estimator as in \cite{MuellerJennyMeyerTwoPhase}:
\begin{align}\label{def:sigmaEstimator}
  \hat{\sigma}_l^2 := \frac{1}{M_{\mathrm{tot}_l}-1} \sum \limits_{j=1}^{M_{\mathrm{tot}_l}} \int_D \bigg( \Big( \frac{1}{M_{\mathrm{tot}_l}} \sum \limits_{i=1}^{M_{\mathrm{tot}_l}}
  (U_l^i - U_{l-1}^i)\Big) - (U_l^j - U_{l-1}^j)\bigg)^2~d{x}.
\end{align}

The work required for the simulation of one sample can vary with an uncertain parameter (e.g. when uncertain 
viscosity influences the time-step restriction). Moreover, on high performance computing systems, random variations in 
work can occur between two executions of the same simulation.
In order to account for this uncertainty, we estimate the work $w_l$ on level $l=0,\ldots,L$ 
by the average work per sample 
$U_l^i$, denoted by $w_l^i$, and define the average work by
\begin{align} \label{def:comptimeEstimator}
  \hat{w_l} := \frac{1}{M_{\mathrm{tot}_l}}  \sum \limits_{i=1}^{M_{\mathrm{tot}_l} } w_l^i.
\end{align}
As a matter of fact, $M_l$ from \eqref{eq:Ml} is also only estimated and
we denote the estimator for $M_l$ by $\hat{M_l}$. 
We have now all ingredients together to state in
\algoref{algo:hpMLMC} the classical MLMC algorithm proposed by Giles \cite{Giles2008}.
Based on this algorithm we want to discuss several important aspects of the $hp$-MLMC method.
First, the complexity of the algorithm will be analyzed in \thmref{thm:complexity}.
The choice of the maximum level $L$ will be considered in \rmref{rem:noLevels}. 
The discussion of the number of warm-up samples $K_0,\ldots, K_L$ (line two in the algorithm), resp. the 
additional samples (lines six and seven in the algorithm) will be postponed to \secref{sec:confIntervals}, where
we derive lower confidence bounds for the optimal number of samples ${M_l}$.
\renewcommand{\thealgorithm}{$hp$-\normalfont{MLMC}}
\begin{algorithm}
  \caption{}
  \begin{algorithmic}[1]
     {\small
    \STATE Fix a tolerance $\epsilon>0$, the maximum level $L\in\N $ and set $\cL :=\{ 0,\ldots,L\}$
    \STATE Compute $K_l$ (warm-up) samples on level $l=0,\ldots,L$ and set $M_{\mathrm{tot}_l}:=K_l$
    \WHILE{$\cL \neq \emptyset$}
    \FOR{$l\in \cL$}
    \STATE  Estimate $w_l$ by \eqref{def:comptimeEstimator}, $\sigma_l^2$ by \eqref{def:sigmaEstimator} and then ${M_l}$
    by \eqref{eq:Ml}
    \IF{$M_l>M_{\mathrm{tot}_l}$}
    \STATE{Add $(M_l-M_{\mathrm{tot}_l})$ new samples of $U_l^i- U_{l-1}^i$ and update $M_{\mathrm{tot}_l}$}
    \ELSE
    \STATE {Set $\cL:= \cL \setminus \{l\}$}
    \ENDIF
    \ENDFOR
    \ENDWHILE
    \STATE Compute $\EMLMC{L}{U_L}$ }
  \end{algorithmic}
  \label{algo:hpMLMC}
\end{algorithm}
\subsection{Computational Complexity of the $hp$-MLMC Method}
In what follows we use the notation $\tilde{q}_l=q_l+1$ for $l\in \N$.\\ 
For the analysis of the  computational complexity of \algoref{algo:hpMLMC} in \thmref{thm:complexity} below
we impose the following assumptions.

\begin{enumerate}[label=(A\arabic*),ref=A\arabic*]
  \item Asymptotic work: \quad $\exists ~\gamma_1 , c_1 >0$ (independent of $h_l, {q}_l$): 
  $w_l\leq  c_1 (h_l^{-1}\tilde{q}_l)^{\gamma_1}$ 
  for all $l\in \N$.\label{ass:BoundedCost}
  \item Bias reduction:\quad $ \exists ~ \kappa_1, c_2 >0$ (independent of $h_l, {q}_l$):  $\|\E{U}-\E{U_l}\|_{L^2(D)} \leq c_2 h_l^{\kappa_1  \tilde{q}_l}$
   for all $l\in \N$. \label{ass:Biasreduction}
  \item Variance reduction between two levels: \quad
    $\exists ~c_3>0$, (independent of $h_l, {q}_l$):
    $ \sigma_l^2  \leq c_3 h_l^{\kappa_2 \tilde{q}_l}$ for some $\kappa_2>0$
    with $\kappa_1 \tilde{q}_0 \geq {\kappa_2 \tilde{q}_0}/{2}$
    and for all $l\in \N$. \label{ass:Varreduction}
\end{enumerate}
\begin{remark} \label{rem:assumptions}
\begin{enumerate}
  \item  \assref{ass:BoundedCost}   provides a bound  for the  computational work in 
terms of the complete  number of degrees of freedom  on  single levels. 
For the Runge--Kutta Discontinuous Galerkin method in $d=2$ spatial dimensions we have
$h_l^{-2}\tilde q_l^{2}$ spatial degrees of freedom.  This has to be multiplied with  the number of 
time-steps, which is proportional to $h_l^{-1}\tilde q_l$, or 
$h_l^{-2}\tilde q_l^{2}$, depending on
the minimum in the CFL condition \eqref{def:cfl}. Thus, 
the total  computational work asymptotically equals $\mathcal{O}(h_l^{-3}\tilde q_l^{3})$, 
resp. $\mathcal{O}(h_l^{-4} \tilde q_l^{4})$.
Therefore, we
expect the parameter $\gamma_1$ to satisfy $\gamma_1=3$,
or $\gamma_1=4$.
  \item In \assref{ass:Biasreduction} it is stated  that the bias, i.e., the deterministic approximation error,
converges with the order of the DG method. For smooth solutions the order of convergence is
$\mathcal{O}(h_l^{q_l+1})=\mathcal{O}(h_l^{\tilde{q}_l}),$ cf. \cite{ZhangShu2004} and the numerical experiments in \cite{HindenlangGassner2012}.
Therefore, we expect $\kappa_1=1$.
  \item In \assref{ass:Varreduction} we require that the variance decays on all levels  
similar to the bias term. If we consider regular solutions of 
a random differential equation, we expect
$\kappa_2=2$, cf. the discussion in \cite[p. 25]{Nobile2017}.  
\end{enumerate}
\end{remark}
In \thmref{thm:complexity} below we present a  complexity bound  for the $hp$-MLMC method. 
	This result generalizes \cite[Theorem 1]{CliffeScheichl2011} on $h$-refined and \cite[Theorem 3]{AppeloeMOMC} on $p$-refined meshes to
	$hp$-refined mesh hierarchies.\\	We distinguish  between  the two different cases
$\kappa_2 \tilde{q}_0 > \gamma_1$ and $\kappa_2 \tilde{q}_0 < \gamma_1$. For the second case
$\kappa_2 \tilde{q}_0 <\gamma_1$
we need to define  a critical level $L^*\in \N$ such that $\kappa_2 \tilde{q}_{L^*} <\gamma_1 \leq \kappa_2 \tilde{q}_{L^*+1}$.

\begin{theorem}[Complexity of the $hp$-MLMC method]\label{thm:complexity}
  For $\beta \in \N$, and $q_0\geq 0$ let
  $ \{q_l:= q_0 + \beta l \}_{l\in \N}$ be a sequence of DG polynomial degrees. 
  Additionally, we consider a family of  meshes with associated
  mesh size $h_l=\lambda^{-l}h_0$
  for some $h_0 \in (0,1)$ and $\lambda \geq 2$.
  Let $\mathcal{V}_{h_l}^{q_l}$ be the corresponding DG spaces.\\
  Under the  assumptions \eqref{ass:BoundedCost} - \eqref{ass:Varreduction}, 
  there exists a constant $c>0$, such that for any tolerance $0<\epsilon<\min \big(1,2c_2 h_0^{\kappa_1 \tilde{q}_0}\big)$,
  there exists a maximum level $L=L(\epsilon)\in \N$, $L(\epsilon) \geq 2$, and a number of samples $M_l$ on each level 
  $l\in \{0,\ldots,L(\epsilon)\}$,
  such that the root mean square error from \eqref{eq:splittingMSE} satisfies $\textsc{RMSE} \leq \epsilon$
  with the computational complexity 
  \begin{align} \label{complexityestimate}
    \Wtot\leq   
    \begin{cases}
       c \epsilon^{-2},	\qquad &  \kappa_2 \tilde{q}_0 > \gamma_1, \\
      c  \epsilon^{\displaystyle  -2-\frac{\gamma_1 - \kappa_2 \tilde{q}_0}{\min\{\kappa_1 \tilde{q}_L,
     \kappa_1 \tilde{q}_{L^*}\}} } , 
      &  \kappa_2 \tilde{q}_{0} < \gamma_1.
    \end{cases} 
  \end{align}
\end{theorem}

\begin{remark}
\begin{enumerate}
 \item  The complexity bound  for the case $\kappa_2 \tilde{q}_{0} < \gamma_1$ in  \thmref{thm:complexity}  suggests  that  there
 exists a threshold for the asymptotic complexity scaling like  
  $\mathcal{O}(\epsilon^{-2-\frac{\gamma_1 - \kappa_2 \tilde{q}_0}{
 \kappa_1\tilde{q}_{L^*}}})$. With other words, increasing the number of levels  $L$ beyond the critical level $L^*$ does not pay off.
 \item \thmref{thm:complexity} does  not include the borderline case $\kappa_2 \tilde{q}_0 = 
 \gamma_1$.  Following the analysis in  \cite{CliffeScheichl2011}  one can make a saturated choice for all 
 sample numbers $M_l$ scaling with the  maximum level  $L(\epsilon)$.  Then a slight variation of  Case A  in the proof 
 of \thmref{thm:complexity} 
 leads to the  total work estimate
 $W_{\text{tot}} = \mathcal{O}\big(\epsilon^{-2} \log (\epsilon)^2 \big)$. Note that the sample numbers $M_l$ in 
 \thmref{thm:complexity} are chosen for each  single level $l$ independent of $L(\epsilon)$ (see e.g.~\eqref{eq:noSamples}). 
\end{enumerate}
\end{remark}

\begin{proof}[Proof of \thmref{thm:complexity}]
  In the proof 
  we determine   first a natural  number of levels $L=L(\epsilon)\ge 2$ such that the bias term
  $\epsilon_{\mathrm{det}}$ in  \eqref{eq:splittingMSE} is bounded by $\epsilon/2$. 
  Based on this number $L$ we define the  sample number $M_l$ for each level $l \in \{0, \ldots, L\}$, 
  and verify  that  $\epsilon_{\mathrm{stat}}\le \epsilon/2$  and the complexity estimate \eqref{complexityestimate} hold.\\ 
 Using   \assref{ass:Biasreduction} on the bias term and the definition of $h_l$ and $q_l$  yields
  \begin{align}\label{ineq:bias}
   \epsilon_{\mathrm{det}} = {\|\E{U}- \E{U_L}\|}_{L^2(D)} \leq c_2 h_l^{\kappa_1 \tilde{q}_l} = c_2 	(\lambda^{-l}h_0)^{\kappa_1(\tilde{q}_0+ \beta l)}        .
  \end{align}
  To ensure that the latter term is bounded by ${\epsilon}/{2}$ for $l=0,\ldots,L$ it suffices to determine   
   $L$ such  that $P(L) \le 0$ holds with 
  \begin{align} \label{ineq:bias2}
   P(l):= \log(2c_2 \epsilon^{-1}) + \kappa_1(\tilde{q}_0+ \beta l)\big( \log(h_0) - 
    l  \log(\lambda) \big).
  \end{align}
The quadratic polynom $P$ is bounded from above and  monotone decreasing on $ [0,\infty)$.  Since 
   $\epsilon \le 2 c_2 h_0^{\kappa_1 \tilde{q}_0} $ holds  we have $P(0) > 0$. Thus, $P$ vanishes for  some  $\bar L >0 $, and we choose 
  the number of levels  to be  
  \begin{align} \label{eq:noLevels}
   L=L(\epsilon)=\max\{\left \lceil{\bar L}\right \rceil,2\}.
  \end{align} 
We proceed  considering the  two different cases in \eqref{complexityestimate}.

  \textbf{Case A: $\kappa_2 \tilde{q}_0 > \gamma_1$ }\\
 We recall the definition of $S$ from Lemma \ref{lemma:sum}  and set $S_0:= h_0^{-\gamma_1/2}S(h_0^{\kappa_2/2}, \lceil \gamma_1/2 \rceil,\tilde{q}_0   )$. Then  we   choose the number of samples on levels $l=0,\ldots,L$ to be 
  \begin{align} \label{eq:noSamples}
    M_l := \Big\lceil 4 \epsilon^{-2} c_3    S_0    h_l^{(\gamma_1 + \kappa_2 \tilde{q}_l)/2} \tilde{q}_l^{-\gamma_1/2}  \Big\rceil.
  \end{align}
  Using \eqref{eq:noSamples} and  \assref{ass:Varreduction} we obtain for the statistical error
  \begin{align*}
  \epsilon_{\mathrm{stat}}^2 = \sum \limits_{l=0}^{L}  \frac{\sigma_l^2}{M_l}
    & \leq \frac{1}{4} \epsilon^2  S^{-1}_0 \sum \limits_{l=0}^{L}  h_l^{(\kappa_2 \tilde{q}_l-\gamma_1)/2} \tilde{q}_l^{\gamma_1 /2},
  \end{align*}
which implies    by  $\kappa_2\tilde{q}_l - \gamma_1> \kappa_2\tilde{q}_0 - \gamma_1 > 0$,   Lemma \ref{lemma:sum} and the definition of $S_0$  the desired estimate 
 \begin{equation} \label{ineq:baseEstimate}
 \epsilon_{\mathrm{stat}}^2 \leq\frac{1}{4} \epsilon^2 S^{-1}_0 \sum \limits_{l=0}^{\infty}  h_0^{(\kappa_2 \tilde{q}_l-\gamma_1)/2} \tilde{q}_l^{\gamma_1 /2} \leq 
 \frac{1}{4}\epsilon^2.
 \end{equation}
Next we derive a bound for the total work $\Wtot$, see \eqref{def:computationalTime}. Using \assref{ass:BoundedCost} and  the ceiling definition of $M_l$  we obtain
\begin{align}\label{ineq:work1}
  \Wtot
  \leq c_1 \sum \limits_{l=0}^{L} M_l h_l^{-\gamma_1} \tilde{q}_l^{\gamma_1}  \notag
  & \leq    
  c_1 \Big(4\epsilon^{-2} S_0 c_3 \sum \limits_{l=0}^{L} h_l^{(\kappa_2 \tilde{q}_l - \gamma_1)/2} \tilde{q}_l^{\gamma_1/2} +  \sum
  \limits_{l=0}^{L} h_l^{-\gamma_1} \tilde{q}_l^{\gamma_1} \Big)\\
  &=: c_1\big(\text{W}_{\mathrm{tot},1} + \text{W}_{\mathrm{tot},2}\big). 
\end{align}
The term  $\text{W}_{\mathrm{tot},1}$  can be shown to be of order $\epsilon^{-2}$  by the same arguments as before.  
It remains to consider  $\text{W}_{\mathrm{tot},2}$.  We have obviously 
\begin{align} \label{ineq:anotherbound}
  \text{W}_{\mathrm{tot},2}
  & \leq  h_L^{-\gamma_1}\tilde{q}_L^{\gamma_1}(1+L) \leq h_L^{-\gamma_1}\tilde{q}_L^{\gamma_1+1},
\end{align}
with the last inequality following from  $1+L\leq \tilde{q}_0 + \beta L = \tilde{q}_{L}$. 
From \eqref{ineq:hL1} in \lemref{lem:help}, 
which we have moved to the end of this proof,
we deduce
\begin{align} \label{ineq:hL}
  h_{L}^{-\gamma_1} \leq (2 c_2\delta^{-1} \epsilon^{-1} )^{\gamma_1/ \kappa_1 \tilde{q}_{0}}.
\end{align}
On the other hand, using the order bound in  \eqref{ineq:hL1} from \lemref{lem:help} and taking the logarithm yields
\begin{align*}
  \tilde{q}_{L} \leq  \frac{ \log( 2 c_2 \delta^{-1}  \epsilon^{-1} ) } {\log(h_0^{-\kappa_1})} =  \frac{ \log( 2 c_2 \delta^{-1}  )
  } {\log(h_0^{-\kappa_1})}
  +  \frac{ \log(\epsilon^{-1}) } {\log(h_0^{-\kappa_1})} =: \hat{c}_1 + \hat{c}_2 \log (\epsilon^{-1}).
\end{align*}
Thus, $\tilde{q}_{L}$  grows  at most logarithmically in $\epsilon^{-1}$. Therefore we find  a constant $\hat{c}_3>0$ which is independent
of $\epsilon$ and  $L$ such that  the 
algebraic estimate 
\begin{align} \label{ineq:pL}
  \tilde{q}_{L} \leq \hat{c}_3 \epsilon^{- \frac{2 -\frac{\gamma_1}{\kappa_1 \tilde{q}_0}}{\gamma_1+1} }
\end{align}
holds. Note that in Case A the term $2 \kappa_1 \tilde{q}_0 -{\gamma_1}$ is positive  due to 
\assref{ass:Varreduction}.   
Now, using  \eqref{ineq:hL} and
\eqref{ineq:pL} in \eqref{ineq:anotherbound} yields for some $c>0$ independent of $\epsilon$ the bound
\begin{align}\label{eq:secondsumBounded}
   \text{W}_{\mathrm{tot}}  & \leq c  \epsilon^{- \frac{\gamma_1}{\kappa_1 \tilde{q}_0}}\epsilon^{-2 + \frac{\gamma_1}{\kappa_1 \tilde{q}_0}} = c \epsilon^{-2}.
\end{align}
Thus,  Case  A is proven.

\textbf{Case B:  $\kappa_2 \tilde{q}_0 < \gamma_1$}\\
Recall the definition of $L^*\in \N$ to be such that $\kappa_2 \tilde{q}_{L^*} <\gamma_1 \le\kappa_2 \tilde{q}_{L^*+1}  $. 
Let us assume first that we have $L^*< L$ with $L$ from \eqref{eq:noLevels}.\\ 
We choose then  the number of samples $M_l$ for $l=0,\ldots, L^*$ according to 
\begin{align} \label{eq:noSamplesCase2}
  M_l := \Big\lceil 8 \epsilon^{-2} c_3  h_l^{(\gamma_1 + \kappa_2 \tilde{q}_l)/2} h_{L^*}^{-(\gamma_1- \kappa_2\tilde{q}_0)/2} \big(1-
  \lambda^{-(\gamma_1-\kappa_2 \tilde{q}_0)/2)}\big)^{-1}  \Big\rceil,
\end{align}
 and for  $l=L^*+1,\ldots,L$ by 
\begin{align} \label{eq:noSamplesCase2_2}
 M_l := \Big\lceil 8 \epsilon^{-2} c_3 S_*  h_l^{(\gamma_1 + \kappa_2 \tilde{q}_l)/2} \tilde{q}_l^{-\gamma_1/2}  \Big\rceil.
\end{align} 
Here  we used  $S_*= S( h_0^{\kappa_2/2},  \lceil \gamma_1/2 \rceil,\tilde{q}_{L^*+1}) $.
Let us consider a splitting for the statistical error given by
\begin{align}  \label{eq:splittingstatistical}
  \epsilon_{\mathrm{stat}}^2 = \sum \limits_{l=0}^{L}  \frac{\sigma_l^2}{M_l} & = 
   \sum \limits_{l=0}^{L^*}  \frac{\sigma_l^2}{M_l}  + \sum \limits_{l=L^*+1}^{L}  \frac{\sigma_l^2}{M_l} 
   =: \epsilon_{\mathrm{stat},1}^2 + \epsilon_{\mathrm{stat},2}^2. 
\end{align}
Using the definition \eqref{eq:noSamplesCase2} of $M_l$  for $l\in \{ 0,\ldots, L^*\}$    we have for  the first term 
\begin{align*}
    \epsilon_{\mathrm{stat},1}^2
   &\leq \frac{1}{8} \epsilon^2  h_{L^*}^{(\gamma_1- \kappa_2\tilde{q}_0)/2}	
   (1-  \lambda^{-(\gamma_1-\kappa_2 \tilde{q}_0)/2)}) \sum \limits_{l=0}^{{L^*}}
   	h_l^{(\kappa_2 \tilde{q}_l-\gamma_1)/2} \\
  & \leq \frac{1}{8} \epsilon^2  h_{L^*}^{(\gamma_1- \kappa_2\tilde{q}_0)/2} (1- \lambda^{-(\gamma_1-\kappa_2 \tilde{q}_0)/2)}) \sum \limits_{l=0}^{L^*}
  h_l^{(\kappa_2 \tilde{q}_0- \gamma_1)/2}.
\end{align*}
For the last inequality we used the fact that
$h_l^{(\kappa_2 \tilde{q}_l-\gamma_1)/2} <h_l^{(\kappa_2 \tilde{q}_0-\gamma_1)/2}$ holds
as long as $0> \kappa_2 \tilde{q}_l-\gamma_1 > \kappa_2 \tilde{q}_0-\gamma_1 $.
Next we use $h_l=\lambda^{L^*-l} h_{L^*}$ to get  by $\gamma_1-\kappa_2 \tilde{q}_0>0 $ in Case B  and $\lambda \geq2$ the estimate 
\begin{align*}
    \epsilon_{\mathrm{stat},1}^2
   &\leq \frac{1}{8} \epsilon^2  h_{L^*}^{(\gamma_1- \kappa_2\tilde{q}_0)/2} (1- \lambda^{-(\gamma_1-\kappa_2 \tilde{q}_0)/2)}) \sum \limits_{l=0}^{L^*}
  (\lambda^{L^*-l} h_{L^*})^{-(\gamma_1-\kappa_2 \tilde{q}_0)/2} \\ \notag
  & 
  = \frac{1}{8} \epsilon^2 (1- \lambda^{-(\gamma_1-\kappa_2 \tilde{q}_0)/2)}) \sum \limits_{l=0}^{\infty}
      {\big(\lambda^{-(\gamma_1-\kappa_2 \tilde{q}_0)/2} \big)}^l   \leq \frac{1}{8} \epsilon^2. \notag
\end{align*}
Since $M_l$ is defined for $l \ge L^*  $  (see \eqref{eq:noSamplesCase2_2}) almost  as in Case A (see \eqref{eq:noSamples}) (with a level shift expressed in $S_*$)  we can 
apply the same arguments  as in Case A to show   $\epsilon_{\mathrm{stat},2}^2 \leq \epsilon^2/8$. Thus  the splitting  \eqref{eq:splittingstatistical} leads to the desired result 
$\epsilon_{\mathrm{stat}}^2 \leq \epsilon^2/4$
for the complete statistical error.

Similarly to the previous splitting \eqref{eq:splittingstatistical}  we rewrite the total work as 
\begin{align} \label{eq:splitsumwork}
 \Wtot =\sum \limits_{l=0}^{L} M_l w_l= \sum \limits_{l=0}^{L^*} M_l w_l + \sum \limits_{l=L^*+1}^{L} M_l w_l 
       =: \text{W}_{\mathrm{tot},1} + \text{W}_{\mathrm{tot},2}. 
\end{align}
The definition of $M_l$ in \eqref{eq:noSamplesCase2} implies 
\begin{align*} 
  \text{W}_{\mathrm{tot},1} &\leq c_1 \sum \limits_{l=0}^{L^*} \Big(8 \epsilon^{-2}  c_3 h_l^{(\gamma_1 + \kappa_2 \tilde{q}_l)/2}
  h_{L^*}^{-(\gamma_1- \kappa_2\tilde{q}_0)/2} 
  \Big(1- \lambda^{-(\gamma_1-\kappa_2 \tilde{q}_0)/2)}\Big)^{-1}  +1\Big) h_l^{-\gamma_1} \tilde{q}_l^{\gamma_1} \\\notag
  &=:c_1 ( \text{W}_{\mathrm{tot},11} + \text{W}_{\mathrm{tot},12} ).
\end{align*}
Starting with  $  \text{W}_{\mathrm{tot},11}$ we use the definitions of $h_l, q_l$ and estimate
\begin{align*}
 %
\text{W}_{\mathrm{tot},11}  &\leq  8 c_3 \epsilon^{-2}  h_{L^*}^{-(\gamma_1- \kappa_2\tilde{q}_0)/2}  
  \Big(1- \lambda^{-(\gamma_1-\kappa_2 \tilde{q}_0)/2)}\Big)^{-1} 
  \sum  \limits_{l=0}^{L^*}  
   (\lambda^l h_{L^*})^{-(\gamma_1- \kappa_2 \tilde{q}_0)/2} \tilde{q}_l^{\gamma_1}
  \\ 
  & <  8 c_3  \epsilon^{-2} h_{L^*}^{-(\gamma_1- \kappa_2 \tilde{q}_0)} 
  \Big(1- \lambda^{-(\gamma_1-\kappa_2 \tilde{q}_0)/2)}\Big)^{-1} \sum
  \limits_{l=0}^{\infty}  (\lambda^{-(\gamma_1- \kappa_2 \tilde{q}_0)/2})^l  (\tilde{q}_0+\beta l)^{\gamma_1}.
\end{align*} 
Being in Case B and the  condition $\lambda\geq2$ ensures the existence of the infinite sum  
such that we are led for some $c>0$ 
by  \eqref{ineq:hL_2} in \lemref{lem:help} to 
\begin{align} \label{Wtot11}
 \text{W}_{\mathrm{tot},11}  
  \leq c \epsilon^{-2-\frac{\gamma_1 - \kappa_2 
   \tilde{q}_0}{\kappa_1 \tilde{q}_{L^*}}}.
\end{align}
Advancing  to the total work contribution $\text{W}_{\mathrm{tot},12}$  we estimate
\begin{align*}
\text{W}_{\mathrm{tot},12}  =   \sum \limits_{l=0}^{L^*} h_l^{-\gamma_1} \tilde{q}_l^{\gamma_1} 
 \leq  h_{L^*}^{-\gamma_1}\tilde{q}_{L^*}^{\gamma_1} (1+ L^*)       \leq  h_{L^*}^{-\gamma_1}\tilde{q}_{L^*}^{\gamma_1+1}.
\end{align*}
By the same arguments employed  to derive  \eqref{ineq:pL} in Case A, but
using \eqref{ineq:hL_2} from \lemref{lem:help} instead, we find a constant $\hat{c}_4>0$ such that
\begin{align} \label{ineq:pL2}
  \tilde{q}_{L^*} \leq  \hat{c}_4  \epsilon^{\frac{-2 + \frac{\kappa_2 \tilde{q}_0}
  {\kappa_1 \tilde{q}_{L^*}}}{\gamma_1+1} }
\end{align}
holds. In this case we need $\kappa_1 \tilde{q}_{L^*} > {\kappa_2 \tilde{q}_0}/{2}$, which follows
from \assref{ass:Varreduction} because $\kappa_1 \tilde{q}_{L^*}  > \kappa_1 \tilde{q}_{0}
\geq {\kappa_2 \tilde{q}_0}/{2}$.
Altogether we end up with
\begin{align*}
  \text{W}_{\mathrm{tot},12}  \leq  h_{L^*}^{-\gamma_1}\tilde{q}_{L^*}^{\gamma_1+1}
 \leq 	c \epsilon^{-2-\frac{\gamma_1 - \kappa_2 \tilde{q}_0}{\kappa_1 \tilde{q}_{L^*}}},
\end{align*}
which yields with \eqref{Wtot11} the desired bound for    $\text{W}_{\mathrm{tot},1} $ in \eqref{eq:splitsumwork}.\\
For the second term  $\text{W}_{\mathrm{tot},2} $ in \eqref{eq:splitsumwork} we use an index shift and apply the same arguments 
which we have used for  Case A with $\kappa_2 \tilde{q}_0 > \gamma_1$. We then obtain
\begin{align*}
\text{W}_{\mathrm{tot},2}  
 \leq c \epsilon^{-2} 
\leq c \epsilon^{-2-\frac{\gamma_1 - \kappa_2 \tilde{q}_0}{\kappa_1 \tilde{q}_{L^*} }},
\end{align*}
which concludes the proof of the theorem for  Case B and  $L^*<L$.\\
It remains to consider $L\leq L^*$.   Then one defines $M_l$ for $l=1,\ldots,L$ like in \eqref{eq:noSamplesCase2} and gets for the statistical error 
and for the total work only the first sums in \eqref{eq:splittingstatistical}, \eqref{eq:splitsumwork}, respectively.  Exactly the same estimates imply the results for the statistical error and lead to  
\begin{align*}
 \text{W}_{\mathrm{tot}} \leq c
  \epsilon^{ -2-\frac{\gamma_1 - \kappa_2 \tilde{q}_0}{\kappa_1 \tilde{q}_L }},
\end{align*}
which concludes the proof.
\end{proof}
Two auxiliary lemmas  have been used  in  the  proof of \thmref{thm:complexity}.   The first result refers to 
the convergence of some series  and is taken from \cite{AppeloeMOMC}.
\begin{lemma}\cite[Lemma 5.1]{AppeloeMOMC}  \label{lemma:sum}
	For   integers $p,q\geq 1$, there are numbers $d_1,\ldots,d_p \in \R$ such that we have
	for $r\in (0,1)$ 
	\begin{align} \label{eq:finitesum}
	\sum \limits_{l=0}^\infty r^{({q}+ \beta l)}({q}+ \beta l)^p = 
	S(r,p,q)
	\end{align} 
	with
	\begin{equation} \label{defS} 
	S(r,p,q) = 
	\sum \limits_{k=1}^{p} d_k r^k f^{(k)}(r), \quad f(r)
	= \frac{r^{{q}}}{1-r^{\beta}}. 
	\end{equation}
\end{lemma}	
The second  lemma  establishes estimates  on $\epsilon$  exploiting the exact definition of $L$.
\begin{lemma} \label{lem:help} Let  the  assumptions from
	\thmref{thm:complexity} be valid. Let  $L\in \N$, $L\ge2$, be given as in \eqref{eq:noLevels}.
Then, there exists a constant  $\delta>0$ independent
of $\epsilon$, 
 such that
the following inequality holds.
\begin{align} \label{ineq:hL1}
& \frac{\epsilon}{2} \delta  \leq c_2\min \big\{  h_L^{\kappa_1 \tilde{q}_0},  
h_0^{\kappa_1 \tilde{q}_L} \big\}
\end{align}
Moreover, for $L^* \in \N$ satisfying 
$ L^*< L$, we have the estimate
\begin{align} \label{ineq:hL_2}
& \frac{\epsilon}{2} \leq c_2  \min \big\{ h_{L^*}^{\kappa_1\tilde{q}_{0}},
 h_{0}^{\kappa_1\tilde{q}_{L^*}}   \big\}.
\end{align} 
\end{lemma}
\begin{proof}
  A straightforward calculation yields
  \begin{align*}
  h_{L}^{\kappa_1 \tilde{q}_{L}} & = (\lambda^{-1} h_{L-1})^{\kappa_1(\tilde{q}_{L-1}+\beta)} 
  \\ 
  & = h_{L-1}^{\kappa_1 \tilde{q}_{L-1}} h_{L-1}^{\kappa_1 \beta} \lambda^{- \kappa_1 q_L}
  \\
   & = h_{L-1}^{\kappa_1 \tilde{q}_{L-1}} 
   \lambda^{-(L-1)\kappa_1 \beta} h_0^{\kappa_1 \beta}
   \lambda^{- \kappa_1 \tilde{q}_0 } 
   \lambda^{- \kappa_1 \beta L}
   \\
   & = h_{L-1}^{\kappa_1 \tilde{q}_{L-1}} 
   \lambda^{\kappa_1 (\beta-\tilde{q}_0)} h_0^{\kappa_1 \beta} \lambda^{-2 L \kappa_1 \beta} 
   \\
   & =  h_{L-1}^{\kappa_1 \tilde{q}_{L-1}} \delta \lambda^{-2 L \kappa_1 \beta},
  \end{align*}   
  where we have set
$\delta:=\lambda^{\kappa_1 (\beta-\tilde{q}_0)} h_0^{\kappa_1 \beta}$. 
  Let $\bar L\in (0,\infty)$ be the zero of $P$ as defined  in 
  \eqref{ineq:bias2}. For  $\bar L\geq 1$
  we deduce from  $L \ge \bar L \ge L-1 \ge 1$ that
  \begin{align}\label{eq:boundsError}
 \frac{\epsilon}{2}\delta \lambda^{-2 L \kappa_1 \beta} \leq 
c_2 h_{L-1}^{\kappa_1 \tilde{q}_{L-1}} \delta \lambda^{-2 L \kappa_1 \beta}
 = c_2 h_{L}^{\kappa_1 \tilde{q}_{L}}.
  \end{align}   
After rearranging 
\eqref{eq:boundsError} we obtain
\begin{align}\label{eq:somecalc}
\frac{\epsilon}{2}\delta \leq c_2 \lambda^{2 \kappa_1 \beta } h_{L}^{ \kappa_1 \tilde{q}_L}
& =c_2 \lambda^{2 \kappa_1 \beta  L} 
\lambda^{-L \kappa_1 \tilde{q}_L} h_0^{\kappa_1 \tilde{q}_L}
\\
&
 = c_2\lambda^{2 \kappa_1 \beta L} 
\lambda^{-L \kappa_1 \tilde{q}_0 } \lambda^{-\kappa_1 \beta L^2 }  h_0^{\kappa_1 \tilde{q}_L} \notag
\\
& = c_2 \lambda^{\kappa_1 \beta (2L-L^2)} \lambda^{-L \kappa_1 \tilde{q}_0 } h_0^{\kappa_1 \tilde{q}_0} 
h_0^{\kappa_1 \beta L} \notag.
\end{align}
Because of $L\geq 2$ it holds that $(2L-L^2)\leq 0$ and 
thus $\lambda^{\kappa_1 \beta (2L-L^2)} \leq 1$. 
Furthermore, $ h_0^{\kappa_1 \beta L} \leq 1$
and $\lambda^{-L \kappa_1 \tilde{q}_0 } \leq 1$ which  yields  altogether \eqref{ineq:hL1}.
\\
If $\bar L \in (0,1)$ holds  it follows from 
\eqref{eq:noLevels} that $L=2$ and we calculate
\begin{align*}
h_{2}^{\kappa_1 \tilde{q}_{2}} & =
\lambda^{-2 \kappa_1 \tilde{q}_{2}}
h_0^{\kappa_1 (\tilde{q}_0 + 2 \beta)}
= \delta h_0^{\kappa_1 \tilde{q}_0},
\end{align*}
where we have defined
$\delta:= \lambda^{-2 \kappa_1 \tilde{q}_{2}} h_0^{ 2 \kappa_1  \beta}$. Finally, \eqref{ineq:hL1}
follows from the fact that $c_2 h_0^{\kappa_1 \tilde{q}_0} \geq  \frac{\epsilon}{2}$.
The estimate \eqref{ineq:hL_2}  for $L^* < L$  follows from
\begin{align*}
c_2 h_{L^*}^{\kappa_1\tilde{q}_{L^*}} \geq c_2 h_{L-1}^{\kappa_1\tilde{q}_{L-1}}
\geq \frac{\epsilon}{2}.
\end{align*}
\end{proof}
\begin{remark} [Choice of the maximum level] \label{rem:noLevels}
  The number of levels $L$ in \algoref{algo:hpMLMC}  can be computed a priori 
  using \eqref{eq:noLevels}. It is also possible
  to compute $L$ on the fly.
  To this end we consider, similarly as in \cite{Giles2008,AppeloeMOMC},
  \begin{flalign*}
     \|\E{U}- \E{U_L}\|_{L^2(D)}  &= \Big\|\sum \limits_{l=L+1}^\infty(\E{U_l}- \E{U_{l-1}})\Big\|_{L^2(D)}
    \\ &\leq  \|\E{U_L}- \E{U_{L-1}}\|_{L^2(D)} \sum \limits_{l=L+1}^\infty \frac{\|\E{U_l}-
    \E{U_{l-1}}	\|_{L^2(D)}}{\|\E{U_L}- \E{U_{L-1}}\|_{L^2(D)}}
    \\ &\leq  \|\E{U_L}- \E{U_{L-1}}\|_{L^2(D)} \sum \limits_{l=1}^\infty (\lambda^{-\kappa_1 \tilde{q}_0} h_0^{\kappa_1 \beta})^l
    \\ & =  \|\E{U_L}- \E{U_{L-1}}\|_{L^2(D)} \frac{\lambda^{-\kappa_1 \tilde{q}_0} h_0^{\kappa_1 \beta}}{1-(\lambda^{-\kappa_1 \tilde{q}_0}
    h_0^{\kappa_1 \beta})},
  \end{flalign*}
  where we have used \assref{ass:Biasreduction}.
  Therefore, the condition for adding new levels becomes
  \begin{align} \label{eq:addLevel}
    \max \limits_{j\in \{0,1,2\}} \frac{(\lambda^{-\kappa_1 \tilde{q}_0} h_0^{\kappa_1 \beta})^{(j+1)}}{1-(\lambda^{-\kappa_1 \tilde{q}_0}
    h_0^{\kappa_1 \beta})} \|\E{U_{L-j}}- \E{U_{L-j-1}}\|_{L^2(D)}  \leq \frac{1}{2} \epsilon.
  \end{align}
  This criterion ensures that  the deterministic error approximated by an extrapolation
  from the three finest meshes is within the desired range, cf. \cite{AppeloeMOMC}.
\end{remark}
%
\subsection{Confidence Intervals for the Number of Additional Samples} \label{sec:confIntervals}
In this section we discuss the computation of the optimal number of samples $M_l$ based on
confidence intervals, having in mind the use of queue-based HPC systems.
In most modern large-scale computing systems, access to compute nodes is based on job schedulers. 
For the execution of a job, a certain number of nodes can be requested for a specified time slot. 
The job is executed after some queuing time, which can be much longer than the actual job execution time.
In the context of $hp$-MLMC, it is advisable to submit a new job to the queue for each iteration of \algoref{algo:hpMLMC}, 
since 
otherwise idle times of the compute nodes are very difficult to avoid.
As each iteration requires its own queuing time, 
it is our aim to compute as many new samples as possible during one iteration of 
\algoref{algo:hpMLMC}. On the other hand, we want  
to avoid computing more samples than optimal, as this would decrease 
the efficiency of the $hp$-MLMC method and from an economical point of
view wasted computing time is expensive. In that sense, one is facing two competing issues,
namely either  to reduce queuing time
by computing as many samples as possible per iteration or to reduce
the number of unnecessarily computed samples, i.e. saving computing time. 
A straightforward approach to satisfy the first issue is to rely
on the standard estimator $\hat{M}_l$.
However, this approach contradicts the second aim of saving computing time. 
Let us recall that the quantities $w_l$ and $\sigma_l$
in \eqref{eq:statisticalError} and \eqref{def:computationalTime} are not known
exactly but are estimated 
by $\hat{w}_l,\hat{\sigma}_l$.
In most cases the number of samples in the warm-up phase and after the first iteration 
of the $hp$-MLMC algorithm is too small to obtain a reliable estimate $\hat{M}_l$ of
the optimal number of samples $M_l$.
This wrong estimate may then lead to an
severe overestimation of ${M}_l$ (see for example \figref{fig:mloptmin_a})
and thus spoils the goal of avoiding
the computation of unnecessary samples and saving computing time.
%
In order to satisfy the second goal of saving computing time, we want
to properly account for the fact that $M_l$ is only estimated 
by constructing a confidence interval for $M_l$.

More specifically, we want to construct a one-sided confidence interval
$\I_{{M}_l}= [\underline{{M}_l},\infty)$, such that
$\P(M_l \in \I_{{M}_l})\ge 1-\alpha $.
To obtain the desired confidence interval we construct corresponding
one-sided confidence intervals for $\sigma_l$, $w_l$ denoted by
$\I_{\underline{\sigma_l}}= [\underline{\sigma_l},\infty), \quad\I_{\underline{w_l}}=[\underline{w_l},\infty),\quad \I_{\overline{w_l}}=(-\infty,\overline{w_l}],$ respectively.
As we do not have any information about the underlying distributions of
$\sigma_l$ and $w_l$ we construct the confidence interval
based on asymptotic confidence intervals and hence our approach is heuristic because 
the Central Limit Theorem implies that the number of samples needs to be sufficiently 
large to ensure that the estimators are asymptotically normally distributed. This seems to be a contradiction 
to the fact that we choose a small number of samples for the warm-up phase.
However, we show in estimate \eqref{ineq:probEst} that our construction of the confidence interval 
is very conservative and yields a robust lower estimate on the optimal number of samples, although 
the number of warm-up samples scales like $\mathcal{O}(1)$. Indeed, we never 
overestimated the optimal number of samples in our computations, justifying 
our approach.

For the construction of the confidence interval for $\sigma_l$ we use the  method described in \cite[Formula (6)]{banik2014},
which employs an adjustment to the degrees of freedom of the $\chi^2$-distribution.
More precisely, we let 
\begin{align*} 
&\hat{r}_l:=\frac{2 M_{\mathrm{tot}_l} }{\hat{\gamma}_{e_l}+ \big(\frac{2M_{\mathrm{tot}_l}}{M_{\mathrm{tot}_l}-1}\big)}, \\
&\hat{\gamma}_{e_l} = \frac{M_{\mathrm{tot}_l}(M_{\mathrm{tot}_l}+1)}{(M_{\mathrm{tot}_l}-1)(M_{\mathrm{tot}_l}-2)(M_{\mathrm{tot}_l}-3)} \frac{\hat{\mu}_l^4}{\hat{\sigma}_l^4} - \frac{3(M_{\mathrm{tot}_l}-1)^2}{(M_{\mathrm{tot}_l}-2)(M_{\mathrm{tot}_l}-3)},
\end{align*}
where $ \hat{\mu}_l^4 :=  \| \sum \limits_{i=1}^{M_{\mathrm{tot}_l}} \Big((U_l^i-U_{l-1}^i)-\frac{1}{M_{\mathrm{tot}_l}} \sum \limits_{j=1}^{M_{\mathrm{tot}_l}} (U_l^j-U_{l-1}^j) \Big)^4\|_{L^2(D)}$. For the lower confidence interval $\I_{\underline{\sigma_l}}= [\underline{\sigma_l},\infty)$ we 
therefore define
\begin{align} \label{eq:confSigma}
\underline{\sigma_l}:=\sqrt{\frac{\hat{r}_l \hat{\sigma}_l^2}{\chi_{1-\frac{\alpha}{2},\hat{r}_l}^2}},
\end{align}
where for some $\alpha \in (0,1)$, $\chi_{1-\frac{\alpha}{2},\hat{r}_l}^2$ is the $(1-\frac{\alpha}{2})$-quantile 
of the $\chi^2$-distribution for $\hat{r}_l$ degrees of freedom.
If the random samples are normally distributed, it follows that $\hat{\gamma}_{e_l}=0$ and thus $\hat{r}_l=M_{\mathrm{tot}_l}-1$, i.e. we obtain 
the standard confidence interval for the variance of a normal distribution (cf. \cite{banik2014}).
For $w_l$ we compute the confidence interval using the standard asymptotic confidence interval for the mean, i.e.
\begin{align} \label{eq:confW} 
\underline{w_l}:=\hat{w}_l- z_{1-\frac{\alpha}{2}} \frac{\hat{\sigma}_{w_l}}{\sqrt{M_{\mathrm{tot}_l}}}, \quad \overline {w_l} := \hat{w}_l+ z_{1-\frac{\alpha}{2}} \frac{\hat{\sigma}_{w_l}}{\sqrt{M_{\mathrm{tot}_l}}}, 
\end{align}
where $z_{1-\frac{\alpha}{2}}$ is the $(1-\frac{\alpha}{2})$-quantile of the normal distribution and $\hat{\sigma}_{w_l}^2$ is the unbiased estimator 
for the variance of $w_l$.
We then define
\begin{align} \label{eq:Mloptmin}
  \underline{{M}_l} =
  \frac{1}{\epsilon^2}\frac{\underline{\sigma_l}}{\sqrt{\overline{w_l}}}
  \left(
  \sum \limits_{k=0}^{L} \underline{\sigma_k}\sqrt{\underline{w_k}}
  \right)
\end{align}
and the confidence interval $\I_{\underline{M_l}}:=[\underline{{M}_l} , \infty)$. Moreover, for  $l=0,\ldots,L$  we define the events
\begin{align*}
  X_l			    &:=\{M_l	  \in \I_{\underline{M_l}}	\},				
  \Sigma_{\epsilon,l,\text{lower}} :=\Big\{ \frac{1}{\epsilon^2}\sigma_l \in \I_{\frac{1}{\epsilon^2}\underline{\sigma_l}} \Big\}, 
  \Sigma_{l,\text{lower}} :=\{\sigma_l \in \I_{\underline{\sigma_l}} \},	   \\ 
  W_{l,\text{lower}}	  &:=\{\sqrt{w_l}      \in \I_{\underline{\sqrt{w_l}}}	    \},  
  W_{l,\text{upper}}	  :=\{\sqrt{w_l}      \in \I_{\overline{\sqrt{w_l}}}\}.
\end{align*}
It then follows that $Y_l \subseteq X_l$, with
$
  Y_l :=
  \bigcap \limits_{k=0}^{L} \Big(  \Sigma_{k,\text{lower}}
  \cap
   W_{k,\text{lower}}
  \cap
  W_{l,\text{upper}}
  \cap
   \Sigma_{\epsilon,l,\text{lower}} \Big),
$
for all $l=0,\ldots,L$.  Using elementary probability estimates and De Morgan's rule we estimate
\begin{align} \label{ineq:probEst}
  \P(X_l) \geq \P(Y_l) =& 1- \P(Y_l^c) \notag
  \\  \geq& 1- \sum \limits_{k=0}^L \Big( (\P(\Sigma_{k,\text{lower}}^c) +
      \P(W_{k,\text{lower}}^c ) + \P(W_{l,\text{upper}}^c) +
       \P(\Sigma_{\epsilon,l,\text{lower}}^c ) \Big).
\end{align}
We construct the confidence intervals
$\I_{\underline{\sigma_k}}= [\underline{\sigma_k},\infty)$,
$\I_{\underline{\sqrt{w_k}}}=[\underline{\sqrt{w_k}},\infty)$
and
$\I_{\overline{\sqrt{w_l}}}=(-\infty,\overline{\sqrt{w_l}}]$
such that
$$  \P(\Sigma_{\epsilon,l,\text{lower}} )=\P (\Sigma_{k,\text{lower}})= \P (W_{k,\text{lower}})=\P (W_{l,\text{upper}})=
  1- \frac{\alpha}{4L}.$$
This choice yields $\P(X_l) \geq 1-\alpha,$ for all $l=0,\ldots,L$.
Whereas the lower confidence bound $\underline{M_l}$ helps in saving computing time,
it increases the number of queuing operations.
Therefore, in order to balance the two competing issues of either reducing queuing time or
saving computing time we introduce the parameter $\zeta \in [0,1]$ 
and let 
\begin{align} \label{eq:choiceMl}
\tilde{M}_l := \zeta \underline{M_l} + (1-\zeta) \hat{M}_l.
\end{align}
By setting $\zeta=0$ we pursue the aim of reducing queuing time and
by setting $\zeta=1$ we try to save computing time.
Choosing $\zeta \in (0,1)$ corresponds to finding a strategy in between.
We choose $\zeta$
adaptively for each iteration of 
\algoref{algo:hpMLMC} and for each level $l=0,\ldots,L$. Hence, for each level $l=0,\ldots,L,$ 
$\zeta= \zeta_l^{\mathrm{iter}_l}$ has its own iteration counter $\mathrm{iter}_l$. We choose the following strategy which tries
to realize a trade-off between reducing queuing time and saving computing time:
\begin{align} \label{def:caseszeta}
\zeta_l^{\mathrm{iter}_l}= 
\begin{cases}
1, \quad &\text{ for } \mathrm{iter}_l=1, \\
0.5, \quad &\text{ for } \mathrm{iter}_l= 2 , \\
0, \quad   & \text{ for } \mathrm{iter}_l> 2 .\\
\end{cases}
\end{align}

\begin{algorithm}
  \caption{} 
  \begin{algorithmic}[1]
    {\small
    \STATE Fix a tolerance $\epsilon>0$, set $L=2$ and set $\cL :=\{ 0,\ldots,L\}$ 
    \STATE Compute $K_l$ (warm-up) samples on $l=0,\ldots,L$, set $M_{\mathrm{tot}_l}:=K_l$
    and $\mathrm{iter}_l=1$
    \WHILE{$\cL \neq \emptyset$}
    \FOR{$l\in \cL$}
    \STATE  Estimate $\hat{w}_l$ by \eqref{def:comptimeEstimator}, $\hat{\sigma}_l^2$ by \eqref{def:sigmaEstimator} and
     then $\hat{M}_l$ by \eqref{eq:Ml}
    \STATE  Estimate  $\underline{w_l}$, $\overline{w_l}$ 
    by \eqref{eq:confW}, $\underline{\sigma_l^2}$ 
    by \eqref{eq:confSigma} and then $\underline{M_l}$ by \eqref{eq:Mloptmin}
    \IF{$M_{\mathrm{tot}_l}<\hat{M}_l$}
    \STATE{Set $\zeta_l^{\mathrm{iter}_l}$ according to \eqref{def:caseszeta} and
    compute $\tilde{M_l}$ by \eqref{eq:choiceMl} } 
    \STATE{Add $\big\lceil\tilde{M_l}- M_{\mathrm{tot}_l}\big\rceil$ new samples 
    of $U_l^i- U_{l-1}^i$ }
    \STATE{Set  $M_{\mathrm{tot}_l}:=M_{\mathrm{tot}_l} + 
    \big\lceil\tilde{M_l}- M_{\mathrm{tot}_l}\big\rceil$} 
    \STATE{Set $\mathrm{iter}_l:=\mathrm{iter}_l+1$ }
    \ELSE{}
    \STATE{Set $\mathrm{iter}_l:=\mathrm{iter}_l+1$ }
    \STATE{Skip level $l$}
    \ENDIF
    \ENDFOR
    \IF{ the statistical tolerance \eqref{def:mininmization} is satisfied }
      \IF{  the bias term satisfies \eqref{eq:addLevel} }
        \STATE{Set $\cL:=\emptyset$}
      \ELSE{}
        \STATE{Set $\cL:=\cL \cup \{L+1\}$}
        \STATE{Compute $K_{L+1}$ (warm-up) samples and set $M_{\mathrm{tot}_{L+1}}:=K_{L+1}$}
        \STATE Set $\mathrm{iter}_{L+1}=1$
        \STATE Set $L:=L+1$
      \ENDIF   
    \ENDIF 
    \ENDWHILE  
    \STATE Compute $\EMLMC{L}{U_L}$  }
  \end{algorithmic} 
  \label{algo:hpMLMCConfInterval}
\end{algorithm} 

In the  first iteration we rely on the lower confidence bound $\underline{M_l}$ from 
\eqref{eq:Mloptmin} to avoid overestimating $M_l$. In the second iteration we start to approach 
$\hat{M}_l$  but still consider $\underline{M_l}$ as a safe-guard for overestimating $M_l$. After the second iteration, where $M_{\mathrm{tot}_l}$ is sufficiently 
large, we trust the estimate $\hat{M}_l$. It might happen that during the first two iterations
we have
$\underline{M_l}\leq M_{\mathrm{tot}_l} < \hat{M}_l$. In that case we set $\zeta_l^{\text{iter}_l}
=0$.

The modified $hp$-MLMC method which adds new levels adaptively based on the bias estimate \eqref{eq:addLevel},
is summarized in \algoref{algo:hpMLMCConfInterval}. For the number of warm-up samples $K_l$ 
we typically choose ten to one hundred samples on the coarse levels and three samples on the fine levels. When we add a new level we set the number of warm-up samples also equal to three.
\renewcommand{\thealgorithm}{$hp$-\normalfont{MLMC} with confidence intervals}


\section{Numerical Experiments} \label{sec:numerics}
We present numerical results for the $hp$-MLMC method as introduced in \algoref{algo:hpMLMCConfInterval}.
In \secref{exp:smooth}, we apply the
$h$-, $p$-, $hp$-MLMC method to a smooth benchmark problem to 
verify \thmref{thm:complexity}.
In \secref{exp:cavity}, we apply the $h$-, $p$- and  $hp$-MLMC method to an open cavity flow problem,
an important flow problem from computational acoustics.
We give a detailed comparison of all methods 
and verify that for both problems, $h$-, $p$- and $hp$-MLMC yield 
an optimal asymptotic work. 
This shows that all three methods 
under consideration are applicable for UQ of complex
engineering problems in computational fluid dynamics.
The computations for the second numerical experiment
were performed on Cray XC40 at the High-Performance Computing
Center Stuttgart. The numerical solver relies 
on the Discontinuous Galerkin Spectral Element solver FLEXI \cite{HindenlangGassner2012}.
The time-stepping uses a Runge--Kutta method of order four.
%

\subsection{Smooth Benchmark Solution} \label{exp:smooth}
In this numerical example we verify \thmref{thm:complexity} by means of 
a smooth manufactured solution, given by 
\begin{align} \label{eq:smoothnse}
 \begin{pmatrix} 
 \rho(t,x_1,x_2,y) \\ (\rho v_1)(t,x_1,x_2,y)\\(\rho v_2)(t,x_1,x_2,y)\\E(t,x_1,x_2,y))
 \end{pmatrix}=
 \begin{pmatrix} 
  2 + A \sin(4 \pi (( x_1+x_2)- ft )) \\
  2 + A \sin(4 \pi (( x_1+x_2)- f t )) \\
  2 + A \sin(4 \pi (( x_1+x_2)- f t )) \\
  \big(2 + A \sin(4 \pi (( x_1+x_2)- f t )) \big)^2
 \end{pmatrix}.
\end{align}
The benchmark solution \eqref{eq:smoothnse} 
is obtained by introducing an additional source term in \eqref{eq:NSE}.
We choose the amplitude  and frequency of \eqref{eq:smoothnse} to be uncertain, i.e.~we let $A \sim \mathcal{U}(0.1,0.9)$
and $\ensuremath{f} \sim \mathcal{N}(1,0.05^2)$.
The spatial domain is $D = (-1,1)^2$ and we consider periodic boundary conditions.
The setup of the mesh hierarchies for $h$-MLMC, $p$-MLMC and $hp$-MLMC can be found in
\tabref{tab:smoothSetup}.
The final computational time for this example
is $T=1$ and the QoI is the  momentum in $x_1$-direction at final time $T$, i.e.
$Q(U)= (\rho v_1)(T,\cdot,\cdot,\cdot).$ For the confidence intervals from \secref{sec:confIntervals} we set
$\alpha$ to be 0.025.
\begin{table}
  \centering
  \begin{tabular}{cccccccccc} %
    \toprule
    \multirow{3}{*}{
      \parbox[c]{.2\linewidth}{\centering level}}
    &
      \multicolumn{2}{c}{$h$-MLMC} &&
      \multicolumn{2}{c}{$p$-MLMC} &&
      \multicolumn{2}{c}{$hp$-MLMC}
    \\
    \cmidrule{2-3} \cmidrule{5-6} \cmidrule{8-9}

    & {\centering $N_l$} & {$q_l$} && {$N_l$} & {$q_l$} && {$N_l$} & {$q_l$} \\
    \midrule
    0 &   16    & 5  && 256   & 3  &&  16    & 3  \\
    1 &   64    & 5  && 256   & 4  &&  64    & 4 \\
    2 &   256   & 5  && 256   & 5  &&  256   & 5  \\
    3 &   1024  & 5  && 256   & 6  &&  1024  & 6  \\
    \bottomrule
  \end{tabular}
  \caption{Level setup for $h$-, $p$- and $hp$-MLMC. Example \ref{exp:smooth}.}
  \label{tab:smoothSetup}
\end{table}

In \tikzfigref{fig:smooth_plots_a} we plot the estimated bias, i.e. the quantity 
$\|\E{U_l}-\E{U_{l-1}}\|_{L^2(D)}$ from \rmref{rem:noLevels}, where
the quantities $\E{U_l}$ and $\E{U_{l-1}}$ have been estimated by the standard MC estimator
\eqref{eq:numExp}.
In \tikzfigref{fig:smooth_plots_b} we plot the bias term 
vs. $h_l^{\tilde{q}_l}$ in a log-log plot. This allows us to estimate $\kappa_1$ from \assref{ass:Biasreduction} using a  linear fit over all data points.
For $h$-MLMC, where we consider a DG polynomial degree of five,
we estimate $\kappa_1 \approx 0.86$ 
and this is in accordance with the 
expected value of one, cf. \rmref{rem:assumptions} 2. On the other hand, the values of $\kappa_1$ for
$p$- and $hp$-MLMC are clearly smaller than one.
\tikzfigref{fig:smooth_plots_c} shows the estimated level variances $\hat{\sigma}^2_l$ 
and its 95\% confidence interval. To verify the assumptions of 
\thmref{thm:complexity}  
we estimate $\kappa_2$ from \assref{ass:Varreduction} in \tikzfigref{fig:smooth_plots_d} by linearly fitting the last three data points.
For $h$-MLMC we estimate $\kappa_2 \approx 1.71$. 
According to \rmref{rem:assumptions} 3. we 
expect $\kappa_2 \approx 2$. Again,  
$p$- and $hp$-MLMC yield  distinctively smaller values than two. However, all three methods
satisfy $ 2\kappa_1 \geq \kappa_2$ from \assref{ass:Biasreduction}.
Next, we check \assref{ass:BoundedCost} for all three methods under consideration, where
we estimate the average work using the sample mean
from \eqref{def:comptimeEstimator}. Since we expect $\gamma_1=3$ 
and because DOF$_l$:=$(h_l^{-1}\tilde{q}_l)^2$, the average work should scale as 
$\hat{w}_l= \mathcal{O}(\text{DOF}_l^{3/2})$ (cf. \rmref{rem:assumptions} 1). This  can be
observed  in  \tikzfigref{fig:smooth_work}.
Combining the estimate of $\kappa_2$ from \figref{fig:smooth_plots_d}
and the fact that $\gamma_1\approx 3$,
we compute that $h$-, $p$- and $hp$-MLMC 
satisfy $\kappa_2 \tilde{q}_0> \gamma_1$ (see the parameters of the coarsest level given in 
\tabref{tab:smoothSetup}). Therefore,
according to \thmref{thm:complexity}
all three methods should yield an 
optimal asymptotic work of $\mathcal{O}(\epsilon^{-2})$, which can be observed in
\tikzfigref{fig:smooth_plots_e}. 
It appears that for this example
 $h$-MLMC is more efficient than $p$- and $hp$-MLMC. This is probably
due to the very good variance reduction 
across all levels and a similar average work 
compared to $hp$-MLMC.

Finally, we check whether the prescribed tolerance $\epsilon$ is
reached by all three methods. To this end we evaluate
the statistical error $\epsilon_{\mathrm{stat}}$ by \eqref{def:sigmaEstimator} and compute 
the deterministic approximation error $\epsilon_{\mathrm{det}}$
using the exact solution \eqref{eq:smoothnse}.
Both terms $\epsilon_{\mathrm{det}}$ and  $\epsilon_{\mathrm{stat}}$
are then summed up as in \eqref{eq:splittingMSE} and plotted in
\tikzfigref{fig:smooth_plots_f}. 
\begin{figure}
  \centering
  \begin{tikzpicture}

\begin{groupplot}[group style={group size=2 by 3,horizontal sep = 1.8cm,  vertical sep = 2cm},
	width=1.\figurewidth,
	height=1.\figureheight,
	scale only axis,
	]

\nextgroupplot[
title = \tikztitle{Estimated bias},
xmode=linear,
xlabel={level},
xtick={0,1,2,3},
xminorticks=true,
ymin=1E-06,
ymode=log,
ymax=1,
ylabel={\scriptsize $\|\E{U_l}-\E{U_{l-1}}\|_{L^2(D)}$},
yminorticks=true,
cycle list name=color,
legend style={font=\scriptsize},
]

\addplot [mark=o, mark options={solid, black}]
  table[row sep=crcr]{%
1		0.11920668E-01 \\
2		0.25624871E-03 \\
3		0.91999214E-05 \\
};
\addlegendentry{$h$}

\addplot [mark=triangle, mark options={solid, black}]
  table[row sep=crcr]{%
1	5.69E-04 \\
2	3.53E-05 \\
3	1.38E-05 \\
};
\addlegendentry{$p$}

\addplot  [mark=diamond*,color=black, mark options={solid, black}]
  table[row sep=crcr]{%
1 0.05949251 \\
2 0.0014448658 \\ 
3 9.992908E-06 \\
};
\addlegendentry{$hp$}

\nextgroupplot[%
title = \tikztitle{Estimated $\kappa_1$},
xminorticks=true,
ymode=log,
xmode=log,
xlabel=$h_l^{\tilde{q}_l}$,
ymin=1E-06,
ymax=1,
axis background/.style={fill=white},
legend style={legend pos= north west,legend cell align=left, align=left, draw=white!15!black, font=\scriptsize},
]
\addplot [black,mark=o]
  table[row sep=crcr]{%
3.814697265625E-06	0.011920668 \\
5.96046447753906E-08	0.00025624871 \\
9.31322574615479E-10	9.1999214E-06 \\
};
\addlegendentry{$\kappa_1 = 0.86$}

\addplot [black,mark=triangle]
  table[row sep=crcr]{%
9.5367431640625E-07	0.00056922153 \\
5.96046447753906E-08	3.5345518E-05 \\
3.72529029846191E-09	1.3787933E-05 \\
};
\addlegendentry{$\kappa_1 = 0.67$}

\addplot [mark=diamond*,color=black]
  table[row sep=crcr]{%
3.0517578125E-05	0.05949251 \\
5.96046447753906E-08	0.0014448658 \\
2.91038304567337E-11	9.992908E-06 \\
};
\addlegendentry{$\kappa_1 = 0.63$}

\nextgroupplot[
title = \tikztitle{Estimated variance},
xlabel={level},
ylabel={$\hat{\sigma}_l^2$},
ymin=1e-13,
ymax=1,
ymode=log,
xmode=linear,
xtick={0,1,2,3},
yminorticks=true,
tick pos=both,
legend style={legend pos= north east,legend cell align=left, align=left, draw=white!15!black, font=\scriptsize},
cycle list name=black white,
]

\addplot[mark=o,solid,mark options={solid,scale=0.5},error bars/.cd, y dir=both, y explicit,error bar style={solid}] 
plot coordinates {
(0,0.0018775694) -= (0,8.07204536000007E-06) += (0,8.14253968999998E-06)
(1,9.7162828E-06) -= (0,2.45925588649999E-07) += (0,2.5913306927E-07)
(2,1.3805651E-08) -= (0,1.6841949167E-09) += (0,2.2320321499E-09)
(3,6.4798512E-12) -= (0,1.06701111089E-12) += (0,1.59490754882E-12)        
};
\addlegendentry{$h$}

\addplot[mark=triangle,black,mark options={solid,scale=0.5},dashed,error bars/.cd, y dir=both, y explicit,error bar style={solid}] plot
 coordinates {
(0,0.0018685385) -= (0,8.87265814000001E-06) += (0,8.95831947999985E-06)
(1,1.7446602E-08) -= (0,2.737826161E-09) += (0,3.9997176739E-09)
(2, 1.7738867E-10) -= (0,2.1640203443E-11) += (0,2.8679358509E-11)
(3,1.6506302E-11) -= (0,2.7180265549E-12) += (0,4.0627515741E-12)
};
\addlegendentry{$p$}

\addplot[mark=diamond*,black,mark options={solid,scale=0.5},dash dot,error bars/.cd, y dir=both, y explicit,error bar style={solid}] plot
  coordinates {
        (0,0.001832402) -= (0,1.01972485100001E-05) += (0,1.03128187599999E-05)
        (1,0.00057969744) -= (0,4.01109370399996E-06) += (0,4.12460232700001E-06)
        (2,1.3758563E-07) -= (0,1.6784505031E-08) += (0,8.9791507E-12)
        (3,8.9791507E-12) -= (0,1.47856073659E-12) += (0,2.2100685327E-12)        
};
\addlegendentry{$hp$}

\nextgroupplot[%
title = \tikztitle{Estimated $\kappa_2$},
xminorticks=true,
ymode=log,
xmode=log,
xlabel=$h_l^{\tilde{q}_l}$,
ymin=1E-13,
ymax=1.,
axis background/.style={fill=white},
legend style={legend pos= north west,legend cell align=left, align=left, draw=white!15!black, font=\scriptsize},
]

\addplot [black,mark=o]
  table[row sep=crcr]{
0.000244140625 0.0018775694 \\  
3.814697265625E-06	9.7162828E-06 \\
5.96046447753906E-08	1.3805651E-08 \\
9.31322574615479E-10	6.4798512E-12 \\
};
\addlegendentry{$\kappa_2 = 1.71$}

\addplot [black,mark=triangle]
  table[row sep=crcr]{%
1.52587890625E-05	0.0018685385 \\
9.5367431640625E-07	1.7446602E-08 \\
5.96046447753906E-08	1.7738867E-10 \\
3.72529029846191E-09	1.6506302E-11 \\
};
\addlegendentry{$\kappa_2 = 1.26$}

\addplot [mark=diamond*,color=black]
  table[row sep=crcr]{%
0.00390625	0.001832402 \\
3.0517578125E-05	0.00029350754 \\
5.96046447753906E-08	1.3758563E-07 \\
2.91038304567337E-11	8.9791507E-12 \\
};
\addlegendentry{$\kappa_2 = 1.25$}

\nextgroupplot[
title = \tikztitle{Asymptotic runtime},
xmode=log,
xmin=0.00008,
xmax=0.02,
xminorticks=true,
xlabel style={font=\color{white!15!black}},
xlabel={$\epsilon$},
ymode=log,
ymin=0.0001,
ymax=0.1,
yminorticks=true,
ylabel style={font=\color{white!15!black}},
ylabel={$\epsilon^2\text{W}_{\text{tot}}$},
axis background/.style={fill=white},
legend style={legend pos= north west,legend cell align=left, align=left, draw=white!15!black, font=\scriptsize},
legend columns=3,
cycle list name=black white,
]

\addplot [color=black, mark=o, mark options={solid, black}]
  table[row sep=crcr]{%
0.01	0.003749 \\
0.005	0.00133075 \\
0.001	0.00052322 \\
0.0005	0.0005016125 \\
0.0003	0.0005263902 \\
0.0001  0.000501929600007\\
};
\addlegendentry{$h$}

\addplot [color=black, mark=triangle, mark options={solid, black}]
  table[row sep=crcr]{%
0.01	0.010286 \\
0.005	0.00622275 \\
0.001	0.00440896 \\
0.0005	0.004268005 \\
0.0003	0.0042759279 \\
0.0001  0.00417191779998 \\
};
\addlegendentry{$p$}

\addplot [color=black, mark=diamond*, mark options={solid, black}]
  table[row sep=crcr]{%
0.01	0.005496 \\
0.005	0.00137675 \\
0.001	0.00128922 \\
0.0005	0.001163042500003 \\
0.0003	0.001004308200009 \\
0.0001  0.000978897099999 \\
\\
};
\addlegendentry{$hp$}

\nextgroupplot[
title = \tikztitle{Estimated tolerance},
xmode=log,
xmin=0.00008,
xmax=0.02,
xminorticks=true,
xlabel style={font=\color{white!15!black}},
xlabel={$\epsilon$},
ymode=log,
ymin=0.00005,
ymax=0.1,
yminorticks=true,
ylabel style={font=\color{white!15!black}},
ylabel={tolerance},
axis background/.style={fill=white},
legend style={legend pos= north west,legend cell align=left, align=left, draw=white!15!black, font=\scriptsize},
legend columns= 2,
]

\addplot [color=black, mark=o, mark options={solid, black}]
  table[row sep=crcr]{%
0.01	0.00466357587298 \\
0.005	0.00247093926596 \\
0.001	0.000513889150111 \\
0.0005	0.000268780811483 \\
0.0003	0.000234489305265 \\
0.0001  7.97886465249602E-05 \\
};

\addlegendentry{$h$ }
\addplot [color=black, mark=triangle, mark options={solid, black}]
  table[row sep=crcr]{%
0.01	0.004822974973051 \\
0.005	0.002463957258151 \\
0.001	0.000516592017427 \\
0.0005	0.00026845687078 \\
0.0003	0.000179041309526 \\
0.0001  7.98345456993602E-05 \\
};
\addlegendentry{$p$}

\addplot [color=black, mark=diamond*, mark options={solid, black}]
  table[row sep=crcr]{%
0.01	0.00492971607914 \\
0.005	0.00248923357138 \\
0.001	0.000547987283766 \\
0.0005	0.000298305138301 \\
0.0003	0.000184720978039 \\
0.0001  8.75818458347825E-05 \\
};
\addlegendentry{$hp$ }

\addplot [color=black, dashed,thick]
  table[row sep=crcr]{%
0.01	0.01 \\
0.005	0.005 \\
0.001	0.001\\
0.0005	0.0005\\
0.0003	0.0003\\
0.0001  0.0001\\
};
\addlegendentry{tolerance}
\end{groupplot}

\end{tikzpicture}
  \settikzlabel{fig:smooth_plots_a} \settikzlabel{fig:smooth_plots_b} \settikzlabel{fig:smooth_plots_c}    
  \settikzlabel{fig:smooth_plots_d} \settikzlabel{fig:smooth_plots_e} \settikzlabel{fig:smooth_plots_f}
  \caption{Estimated bias, variance, tolerance and asymptotic runtime. For $\hat{\sigma}_l^2$ we also plot the 95\% confidence
  interval. Example \ref{exp:smooth}.}
  \label{fig:smooth_runtime}
\end{figure}
The computed number of samples on each level for different tolerances is shown 
in \figref{fig:smooth_samples}. 
As expected, we only need a few computations on the fine levels, the majority of the computations is performed on coarse levels.

\figref{fig:ldc_mloptmin} illustrates the advantage of computing the lower confidence bound 
$\underline{M_l}$ from \eqref{eq:Mloptmin}. In this figure we plot the values of
$\underline{M_l}$ from \eqref{eq:Mloptmin}, of $\tilde{M}_l$ from 
\eqref{eq:choiceMl} and of $\hat{M}_l$ from \eqref{eq:Ml},
for each level $l=0,\ldots, 3$, for the first three 
iterations of \algoref{algo:hpMLMCConfInterval}.
For this example we use $hp$-MLMC.
In \tikzfigref{fig:mloptmin_a} we observe that relying on the estimate of $\hat{M}_0$
would have led to an overestimate of the optimal number of samples by more than 8000 samples.
The same holds for level two (\tikzfigref{fig:mloptmin_b}) where we would have overestimated 
the optimal number of samples  by more than 400.
On the other hand, our proposed strategy ensures that we reach the
standard estimator $\hat{M}_l$ in (at most) three iterations.
This shows in particular the advantage of our approach 
because we avoid computing unnecessary samples, hence we save computing time,
while trying to keep the number of 
queuing operations low.

\begin{figure}
  \centering
  \begin{tikzpicture}

\begin{groupplot}[group style={group size=2 by 3,horizontal sep = 2cm,  vertical sep = 1.8cm},
	width=\figurewidth,
	height=\figureheight,
	scale only axis,
	]
\nextgroupplot[
title = \tikztitle{No. of samples, $h$-MLMC},
xlabel={level},
ymode=log,
ymin=1,
ymax=2000000,
ylabel={no. samples},
xtick={0, 1, 2,3},
xmin=-0.5,
xmax = 3.5,
yminorticks=true,
axis background/.style={fill=white},
legend style={at={(1.9,-0.5)}, anchor=south east, legend cell align=left, align=left, draw=white!15!black, font=\scriptsize},
legend columns=3,
cycle list name=black white,
]

\addplot 
  table[row sep=crcr]{%
0 85 \\
1 6 \\
2 3 \\
};
\addlegendentry{$\epsilon=\expnumber{1}{-2}$}

\addplot 
  table[row sep=crcr]{%
0 350 \\
1 9\\
2 3 \\
};
\addlegendentry{$\epsilon=\expnumber{5}{-3}$}

\addplot 
  table[row sep=crcr]{%
0 9453 \\
1 261 \\
2 3 \\
};
\addlegendentry{$\epsilon=\expnumber{1}{-3}$}

\addplot 
  table[row sep=crcr]{
0 36894 \\
1 929 \\
2 9 \\
};
\addlegendentry{$\epsilon=\expnumber{5}{-4}$}

\addplot 
  table[row sep=crcr]{%
0 103021 \\
1 2848 \\
2 22 \\
3 3 \\
};
\addlegendentry{$\epsilon=\expnumber{3}{-4}$}

\addplot 
  table[row sep=crcr]{%
0 827511 \\
1 22851 \\
2 172 \\
3 3 \\
};
\addlegendentry{$\epsilon=\expnumber{1}{-4}$}

\legend{}

\nextgroupplot[
title = \tikztitle{No. of samples, $p$-MLMC},
xlabel={level},
ylabel={no. samples},
ymode=log,
ymin=1,
ymax=2000000,
xtick={0, 1, 2, 3},
xmin=-0.5,
xmax = 3.5,
yminorticks=true,
axis background/.style={fill=white},
legend columns=3,
cycle list name=black white,
]

\addplot 
  table[row sep=crcr]{%
0 71 \\
1 3 \\
2 3 \\
};
\addlegendentry{$\epsilon=\expnumber{1}{-2}$}

\addplot 
  table[row sep=crcr]{%
0	340 \\
1	3 \\
2	3 \\
};
\addlegendentry{$\epsilon=\expnumber{5}{-3}$}

\addplot 
  table[row sep=crcr]{%
0 7784 \\
1 6 \\
2 3 \\
};
\addlegendentry{$\epsilon=\expnumber{1}{-3}$}

\addplot 
  table[row sep=crcr]{%
0 30373 \\
1 19 \\
2 3 \\
};
\addlegendentry{$\epsilon=\expnumber{5}{-4}$}

\addplot 
  table[row sep=crcr]{%
0 84373 \\
1 56 \\
2 9 \\
3 3\\
};
\addlegendentry{$\epsilon=\expnumber{3}{-4}$}

\addplot 
  table[row sep=crcr]{%
0 676889 \\
1 441 \\
2 55 \\
3 5\\
};
\addlegendentry{$\epsilon=\expnumber{1}{-4}$}

\legend{}

\nextgroupplot[
title = \tikztitle{No. of samples, $hp$-MLMC},
xlabel={level},
ymode=log,
ymin=1,
ymax=2000000,
ylabel={no. samples},
xtick={0, 1, 2, 3},
xmin=-0.5,
xmax = 3.5,
yminorticks=true,
axis background/.style={fill=white},
legend style={at={(1.04,-1)}, anchor=south east, legend cell align=left, align=left, draw=white!15!black, font=\small},
legend columns=2,
cycle list name=black white,
]

\addplot 
  table[row sep=crcr]{%
0 112 \\
1 35 \\
2 3 \\
};
\addlegendentry{$\epsilon=\expnumber{1}{-2}$}

\addplot 
  table[row sep=crcr]{%
0 557 \\
1 116 \\
2 3 \\
};
\addlegendentry{$\epsilon=\expnumber{5}{-3}$}

\addplot 
  table[row sep=crcr]{%
0 13078 \\
1 2955 \\
2 24 \\
};
\addlegendentry{$\epsilon=\expnumber{1}{-3}$}

\addplot 
  table[row sep=crcr]{%
0 61329 \\
1 10004 \\
2 82 \\
};
\addlegendentry{$\epsilon=\expnumber{5}{-4}$}

\addplot 
  table[row sep=crcr]{%
0 173399 \\
1 28611 \\
2 202 \\
3 3 \\
};
\addlegendentry{$\epsilon=\expnumber{3}{-4}$}

\addplot 
  table[row sep=crcr]{%
0 1606176 \\
1 236639 \\
2 2982 \\
3 4\\
};
\addlegendentry{$\epsilon=\expnumber{1}{-4}$}

\nextgroupplot[
title = \tikztitle{average work},
scale only axis,
xmode=log,
xminorticks=true,
ymode=log,
xlabel={DOF},
ylabel={work},
ymax=500,
yminorticks=true,
axis background/.style={fill=white},
legend style={at={(1.04,-1)}, anchor=south east, legend cell align=left, align=left, draw=white!15!black, font=\small},
legend columns=2,
]
\addplot [black,solid,mark=triangle] 
  table[row sep=crcr]{%
576 0.0295285714286 \\
2304 0.175033333333\\
9216 1.3654\\
36864 18.9262 \\
};
\addlegendentry{$h$-MLMC}

\addplot  [black,solid,mark=asterisk,] 
  table[row sep=crcr]{%
4096  0.389514285714 \\
6400  0.740133333333 \\
9216  1.36053333333\\
12544 3.023112\\
};
\addlegendentry{$p$-MLMC}

\addplot [black,solid,mark=diamond*] 
  table[row sep=crcr]{%
256 0.01\\
1600 0.075 \\
9216 1.35933333333\\
50176 33.976 \\
};
\addlegendentry{$hp$-MLMC}

\addplot[black,dashed] 
  table[row sep=crcr]{%
256 0.04096\\
50176 112.39424\\
};
\addlegendentry{$\mathcal{O}(\text{DOF}^{3/2})$}

\end{groupplot}

\end{tikzpicture}%
  \settikzlabel{fig:smooth_samples_a} 
  \settikzlabel{fig:smooth_samples_b}
  \settikzlabel{fig:smooth_samples_c} \settikzlabel{fig:smooth_work}
  \caption{Computed number of samples and  average work on each level. Example \ref{exp:smooth}.}
  \label{fig:smooth_samples}
\end{figure}
\begin{figure}
 \centering
 \begin{tikzpicture}
\begin{groupplot}[group style={group size=2 by 2,horizontal sep = 2cm,  vertical sep = 1.8cm},
	width=1.\figurewidth,
	height=\figureheight,
	scale only axis,
	]
	
\nextgroupplot[
title = \tikztitle{Level 0},
bar width=0.3,
xmin=0.5,
xmax=5.5,
xtick={1, 2, 3},
xtick style={draw=none},
xlabel style={font=\color{white!15!black}},
xlabel={iteration},
ymin=0,
ymax=80000,
ylabel style={font=\color{white!15!black}},
ylabel={\scriptsize estimated no. samples},
axis background/.style={fill=white},
legend style={at={(0.95,0.5)},anchor=east,legend cell align=left, align=left, draw=white!15!black},
legend columns=1,
]

\addplot[ybar, draw=black, area legend] table[row sep=crcr] {%
1	36380 \\
2	57492\\
3	52275 \\
};

\addlegendentry{$\underline{M}_0$}

\addplot[ybar, ,thick, draw=red, area legend] table[row sep=crcr] {%
1	36380 \\
2	61330\\
3	63314\\
};

\addlegendentry{$\tilde{M}_0$};

\addplot[ybar ,dashed, draw=black, area legend] table[row sep=crcr] {%
1	71971\\
2	65168\\
3	63314.0\\
};
\addlegendentry{$\hat{M}_0$};
\nextgroupplot[
title = \tikztitle{Level 1},
bar width=0.3,
xmin=0.5,
xmax=5.5,
xtick={1, 2, 3},
xtick style={draw=none},
xlabel style={font=\color{white!15!black}},
xlabel={iteration},
ymin=0,
ymax=11000,
ylabel={\scriptsize estimated no. samples},
ylabel style={font=\color{white!15!black}},
axis background/.style={fill=white},
legend style={at={(0.95,0.5)},anchor=east,legend cell align=left, align=left, draw=white!15!black},
legend columns=1,
]

\addplot[ybar,  , draw=black, area legend] table[row sep=crcr] {%
1	2998 \\
2	8152\\
3	8378\\
};

\addlegendentry{$\underline{M}_1$}

\addplot[ybar, , thick, draw=red, area legend] table[row sep=crcr] {%
1	2999\\
2	8619\\
3	9330\\
};

\addlegendentry{$\tilde{M}_1$};

\addplot[ybar ,dashed, draw=black, area legend] table[row sep=crcr] {%
1	9787\\
2	9085\\
3	9330\\
};
\addlegendentry{$\hat{M}_1$};

\nextgroupplot[
title = \tikztitle{Level 2},
bar width=0.3,
xmin=0.5,
xmax=5.5,
xtick={1, 2, 3},
xtick style={draw=none},
xlabel style={font=\color{white!15!black}},
xlabel={iteration},
ymin=0,
ymax=200,
ylabel style={font=\color{white!15!black}},
ylabel={\scriptsize estimated no. samples},
axis background/.style={fill=white},
legend style={at={(0.95,0.5)},anchor=east,legend cell align=left, align=left, draw=white!15!black},
legend columns=1,
]
\addplot[ybar,  draw=black, area legend] table[row sep=crcr] {%
1	16\\
2	75\\
3	116\\
};
\addlegendentry{$\underline{M}_2$}
\addplot[ybar, thick, draw=red, area legend] table[row sep=crcr] {%
1	16\\
2	105\\
3	156\\
};
\addlegendentry{$\tilde{M}_2$}
\addplot[ybar , dashed,draw=black, area legend] table[row sep=crcr] {%
1	104\\
2	134\\
3	156\\
};
\addlegendentry{$\hat{M}_2$}

\end{groupplot}
\end{tikzpicture}
 \settikzlabel{fig:mloptmin_a} \settikzlabel{fig:mloptmin_b} 
 \settikzlabel{fig:mloptmin_c}  \settikzlabel{fig:mloptmin_d}
 \caption{Estimated $\underline{M_l}$, $\hat{M}_l$ and chosen number of samples $\tilde{M}_l$ from 
 \eqref{eq:choiceMl} for
 different levels for the $hp$-MLMC method.
 The tolerance in this example is $\epsilon=\expnumber{5}{-4}$. The number of warm-up samples was
 (100, 10, 3)  (Example \ref{exp:smooth}).}
 \label{fig:ldc_mloptmin}
\end{figure}
%
\subsection{Open Cavity}\label{exp:cavity}
In this numerical example we investigate the influence of uncertain input parameters on the aeroacoustic feedback of cavity
flows as in \cite{Kuhn2018}. The prediction of aeroacoustic noise is an important branch of research for example in the
automotive industry.
However, due to the large bandwidth of spatial and temporal scales, a high-order numerical scheme with low dissipation and
dispersion error is necessary to
preserve important small scale information and hence it poses a very challenging numerical problem for UQ.
We consider the flow over a two-dimensional open cavity, cf. \figref{fig:cavity_sketch}, using $h$-, $p$- and $hp$-MLMC.

For this flow problem we consider two uncertain parameters. The first uncertain parameter is
the initial condition for pressure, i.e. we let $p^0 \sim \mathcal{N}(1.8,0.01^2)$  be normally
distributed. With this choice 
the Mach number $\text{Ma}=\frac{v_1}{c}$, with $c=\sqrt{\kappa \frac{p}{\rho}}$ being the speed of sound, becomes uncertain.
The initial condition in primitive variables then reads as
$$\Big(\rho^0, v_1^0,v_2^0,p^0\Big)= \Big(1,1,0,p^0(y)\Big).$$
As a second uncertain parameter, we let the boundary layer thickness in front of the cavity,
$\delta_{99}\sim \mathcal{U}(0.28,0.48)$, be uniformly distributed. 
For the boundary conditions at the inlet we employ
Dirichlet boundary conditions in combination with the
precomputed  Blasius boundary layer profile. 
All wall boundaries are modeled as isothermal no-slip walls where the temperature  $\mathcal{T}$
is computed from the ideal gas law $p=\rho R \mathcal{T}$,
where the gas constant for air satisfies $R=287.058$. 
At the end of the cavity we consider a pressure outflow boundary condition, where the pressure is
specified by the initial pressure. Above the cavity we also consider a
pressure outflow boundary condition and we augment the boundary with a sponge zone (cf. \figref{fig:cavity_sketch})
to avoid that artificial reflections reenter the computational domain. Detailed
information about the sponge zone can be found in \cite{Kuhn2018}.
For our QoI we record
the pressure fluctuations $p(t,x,y)$ on top of the cavity 
at $\bar{x}=(x_1,x_2)=(1.57,0)$ over time and then perform the discrete-time Fourier transform (DTFT)
to obtain the sound pressure spectrum at $\bar{x}$, i.e.
$Q(U)= \text{DTFT}\Big(p(\cdot,\bar{x},y)\Big).$
The corresponding $L^2$-norm is then taken in frequency space.
The mesh hierarchies for $h$-, $p$- and $hp$-MLMC can be found in  \tabref{tab:cavitySetup}.
For the confidence intervals from \secref{sec:confIntervals} we set
$\alpha$ to be 0.025.

\begin{table}
 \centering
 \begin{tabular}{ccccccccc} %
   \toprule
   \multirow{2}{*}{
     \parbox[c]{.2\linewidth}{\centering level}}
   & \multicolumn{2}{c}{$h$-MLMC} &&
   \multicolumn{2}{c}{$p$-MLMC} &&
   \multicolumn{2}{c}{$hp$-MLMC} \\
   \cmidrule{2-3} \cmidrule{5-6}  \cmidrule{8-9}

   & {\centering $N_l$} & {$q_l$} && {$N_l$} & {$q_l$} && {$N_l$} & {$q_l$} \\
   \midrule
   0 & 279  & 7  && 1987 & 4  && 279   & 4 \\
   1 & 423  & 7  && 1987 & 5  && 423   & 5 \\
   2 & 957  & 7  && 1987 & 6  && 957   & 6 \\
   3 & 1987 &7   && 1987 & 7  && 1987  & 7 \\
   \bottomrule
 \end{tabular}
 \caption{Level setup for $h$-, $p$ and $hp$-MLMC (Example \ref{exp:cavity}).}
 \label{tab:cavitySetup}
\end{table}
\begin{figure}
 \centering
 \def\svgwidth{6.6cm}
 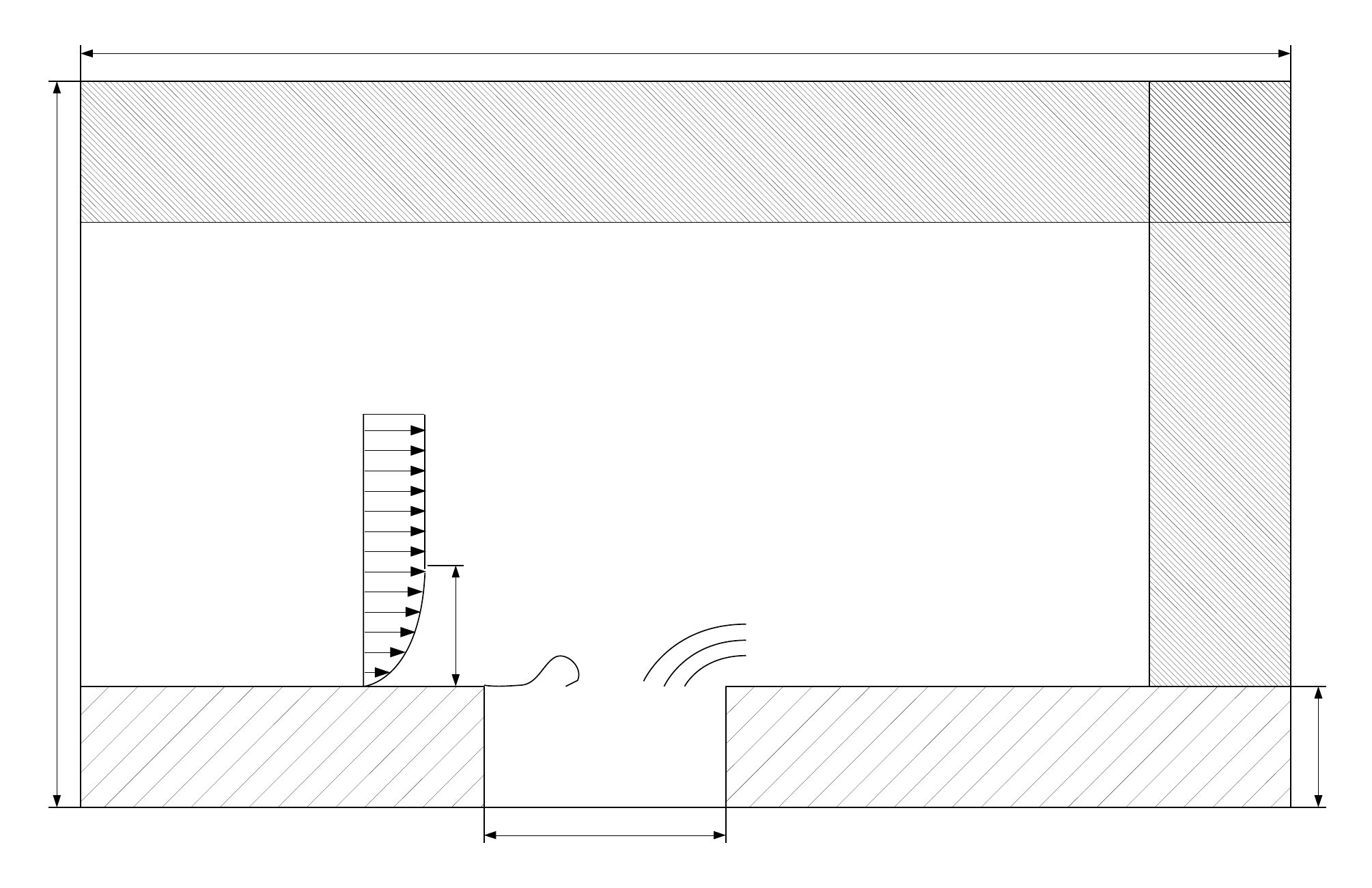
 \includegraphics[width=5.6cm,height=4.5cm]{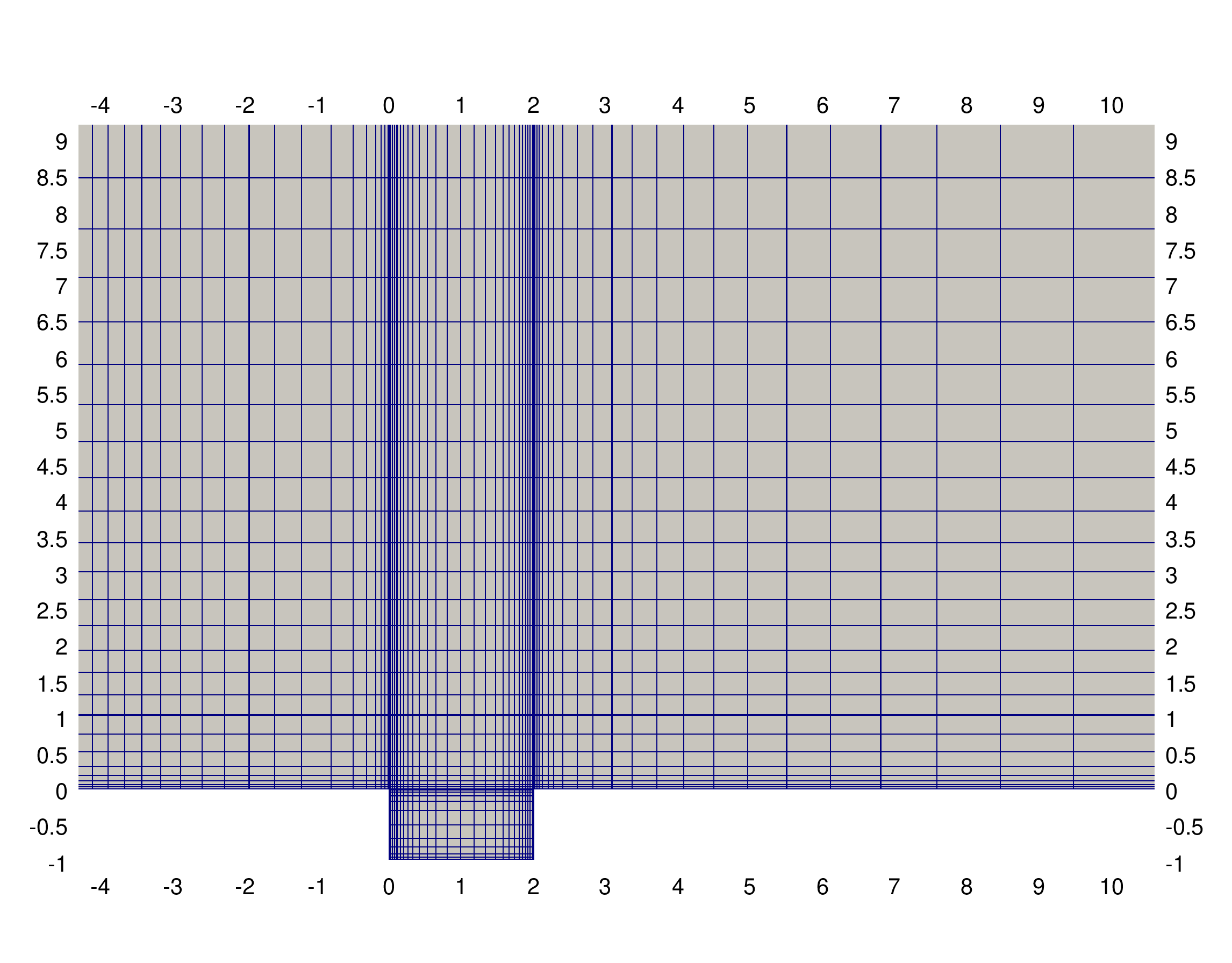}
  \caption{\textit{Left:} Schematic sketch of the open cavity setup with a laminar inflow boundary layer. All geometric
  parameters are adopted from \cite{Kuhn2018} and are non-dimensionalized by the cavity depth. \textit{Right:} Computational
  mesh on the finest level.  (Example \ref{exp:cavity}).  }
  \label{fig:cavity_sketch}
\end{figure}

Considering the bias error we can see that in \tikzfigref{fig:cavity_plots_a} for $p$-MLMC
the bias estimate $\|\E{Q(U_l)}-\E{Q(U_{l-1})}\|_{L^2}$ for $p$-MLMC is smaller than $\expnumber{2.5}{-5}$.
As all three methods share the same finest level, we can assume that the bias error from \eqref{eq:splittingMSE}
is satisfied for all three methods under consideration. 
\tikzfigref{fig:cavity_plots_b} shows the  estimated level variance $\hat{\sigma}_l^2$ across different levels.
All three methods yield a very good variance reduction, especially $p$-MLMC has already a very small variance on level two.
However, the computation of samples on the coarse grids is extremely costly for $p$-MLMC. Taking a closer look at the variance on 
level zero, we see that $hp$-MLMC achieves the same variance as $p$-MLMC but with much less DOFs. 
This yields the computational advantage of $h$- and $hp$-MLMC compared to $p$-MLMC (see \tikzfigref{fig:cavity_plots_c}) for this
open cavity problem.
The  asymptotic work is still optimal for all three methods, which can be seen in \tikzfigref{fig:cavity_plots_c}.
Finally, in \figref{fig:cavity_samples} we plot the number of computed samples for different tolerances
and the average work that is needed to compute one sample on the corresponding level. 
Again, for all three methods most of the computations are performed on the coarse levels as they have a low computational cost.
In contrast to the benchmark problem the average work does not scale as $ \mathcal{O}(\text{DOF}_l^{3/2})$ but more like 
$\mathcal{O}(\text{DOF}_l^{2})$, indicating that $\gamma_1\approx 4$. The main reason for this behavior is that 
the uncertain Mach number influences the time-step size because the speed of sound enters the eigenvalues of the 
Jacobian of the advective fluxes, hence a bigger pressure leads to a smaller time-step size. 
\begin{figure}
 \centering
 \begin{tikzpicture}

\begin{groupplot}[group style={group size=2 by 2,horizontal sep = 2cm,  vertical sep = 1.8cm},
	width=\figurewidth,
	height=\figureheight,
	scale only axis,
	]

\nextgroupplot[
title = \tikztitle{Estimated bias},
xmode=linear,
xlabel={level},
xmin=0.5,
xmax=3.5,
xtick={1,2,3},
xminorticks=true,
ymode=log,
ylabel={\scriptsize $\|\E{Q(U_l)}-\E{Q(U_{l-1})}\|_{L^2}$},
yminorticks=true,
ymax=5e-3,
cycle list name=color,
legend pos=north east,
legend style={legend cell align=left, align=left, draw=white!15!black,font=\scriptsize},
legend columns = 3,
legend pos=north east,
]

\addplot [mark=o, mark options={solid, black}]
  table[row sep=crcr]{%
1		0.74956272E-03 \\
2		0.38697152E-03 \\
3		0.24195132E-03 \\
};
\addlegendentry{$h$}

\addplot [mark=triangle, mark options={solid, black}]
  table[row sep=crcr]{%
1		0.32567585E-04 \\
2		0.30296070E-04 \\
3		2.8434417E-05  \\
};
\addlegendentry{$p$}

\addplot  [mark=diamond*,color=black, mark options={solid, black}]
  table[row sep=crcr]{%
1	0.001176669 \\
2	0.0003209981 \\
3	0.0002289412 \\
};
\addlegendentry{$hp$}

\nextgroupplot[
title = \tikztitle{Estimated variance},
xlabel={level},
ylabel={$\hat{\sigma}_l^2$},
ymode=log,
xmode=linear,
xmin=-0.5,
xmax=3.5,
xtick={0,1,2,3},
yminorticks=true,
ymin=1e-12,
ymax=1e-04,
tick pos=both,
legend style={legend cell align=left, align=left, draw=white!15!black, font=\scriptsize},
legend pos=north east,
cycle list name=black white,
legend columns = 3,
legend pos=north east,
]

\addplot[mark=o,solid,mark options={solid,scale=0.5},error bars/.cd, y dir=both, y explicit,error bar style={solid}] plot coordinates {
        (0,0.00000481)     -= (0,1.1206812697E-007)  += (0,1.1377691112E-007)
        (1,4.347111E-009)  -= (0,3.72435480028E-011) += (0,9.0066645287E-010)
        (2,6.3797549E-010) -= (0,7.1909553949E-011)  += (0,5.8423937777E-010)
        (3,1.8403646E-010)-= (0,8.8216355727E-011)  += (0,0.000000001)        
};
\addlegendentry{$h$}

\addplot[mark=triangle*,black,mark options={solid,scale=0.5},dashed,error bars/.cd, y dir=both, y explicit,error bar style={solid}] plot coordinates {
        (0,4.5715435E-006)  -= (0,1.1206812697E-007)  += (0,1.1788757393E-007)
        (1,1.1930733E-010)  -= (0,3.72435480028E-011) += (0,9.8501346497E-011)
        (2, 1.0938513E-010) -= (0,3.1909553949E-011)  += (0,2.06088020981E-010)
        (3,7.0763225E-011) -= (0,3.39197669255E-011) += (0,3.73964870393E-010)
};
\addlegendentry{$p$}

\addplot[mark=diamond*,black,mark options={solid,scale=0.5},dash dot,error bars/.cd, y dir=both, y explicit,error bar style={solid}] plot coordinates {
        (0,5.52E-006)   -= (0,0.000000061)        += (0,6.24301677200003E-008)
        (1,4.99E-007)  -= (0,1.2515544889E-008)  += (0,1.3181833018E-008)
        (2, 9.99E-010) -= (0,1.86642220413E-010) += (0,2.9872483427E-010)
        (3,4.20E-010) -= (0,1.36307341293E-010) += (0,3.84621783613E-010)        
};
\addlegendentry{$hp$}

\nextgroupplot[
title = \tikztitle{Asymptotic runtime},
xmode=log,
xmin=0.00002,
xmax=0.002,
xminorticks=true,
xlabel style={font=\color{white!15!black}},
xlabel={$\epsilon$},
ymode=log,
ymin=0.01,
ymax=100,
yminorticks=true,
ylabel style={font=\color{white!15!black}},
y tick label style={/pgf/number format/1000 sep=},
ylabel={$\epsilon^2\text{W}_{\text{tot}}$},
axis background/.style={fill=white},
legend style={legend cell align=left, align=left, draw=white!15!black,font=\scriptsize},
legend pos=north west,
legend columns =3,
]

\addplot [color=black, mark=o, mark options={solid, black}]
  table[row sep=crcr]{%
0.001					2.228428 \\
0.0008					1.30228736 \\
0.0005					0.692919 \\
0.0001					0.44536527 \\
0.00008					0.4344829632 \\
0.00005					0.4781698225 \\
};
\addlegendentry{$h$}

\addplot [color=black, mark=diamond, mark options={solid, black}]
  table[row sep=crcr]{%
0.001					10.418502 \\
0.0008					6.27201856 \\
0.0005					5.00066425 \\
0.0001					1.82050242 \\
0.00008					1.7551533952 \\
0.00005					1.69750227 \\
};
\addlegendentry{$p$}

\addplot [color=black, mark=asterisk, mark options={solid, black}]
  table[row sep=crcr]{%
0.001				1.220292\\
0.0008				0.88595264 \\
0.0005				0.36826725\\
0.0001				0.47652013 \\
0.00008				0.4590465984 \\
0.00005				0.5593632175 \\
};
\addlegendentry{$hp$}

\end{groupplot}
\end{tikzpicture}
  \settikzlabel{fig:cavity_plots_a} \settikzlabel{fig:cavity_plots_b} \settikzlabel{fig:cavity_plots_c}
  \settikzlabel{fig:cavity_plots_d}
 \caption{Estimated bias, variance and asymptotic work. For $\hat{\sigma}_l^2$ we also plot the 95\% confidence
 interval  (Example \ref{exp:cavity}).}
 \label{fig:cavity_plot}
\end{figure}
\begin{figure}
 \centering
 \begin{tikzpicture}

\begin{groupplot}[group style={group size=2 by 3,horizontal sep = 1.5cm,  vertical sep = 1.7cm},
	width=\figurewidth,
	height=0.97\figureheight,
	scale only axis,
	]
	
\nextgroupplot[
title = \tikztitle{No. of samples, $h$-MLMC},
xlabel={level},
ylabel={no. samples},
ymode=log,
ymin=1,
ymax=50000,
xtick={0, 1, 2, 3},
xmin=-0.5,
xmax = 3.5,
yminorticks=true,
axis background/.style={fill=white},
legend columns=3,
cycle list name=black white,
]

\addplot 
  table[row sep=crcr]{%
0	22 \\
1	3 \\
2   3 \\
};
\addlegendentry{$\epsilon=\expnumber{1}{-3}$}

\addplot 
  table[row sep=crcr]{%
0	36 \\
1	3 \\
2	3 \\
};
\addlegendentry{$\epsilon=\expnumber{8}{-4}$}

\addplot 
  table[row sep=crcr]{%
0	86 \\
1	6 \\
2	3 \\
};
\addlegendentry{$\epsilon=\expnumber{5}{-4}$}

\addplot 
  table[row sep=crcr]{%
0	2372 \\
1	26 \\
2	6 \\
3	3 \\
};
\addlegendentry{$\epsilon=\expnumber{1}{-4}$}

\addplot 
  table[row sep=crcr]{%
0	3599 \\
1	66 \\
2	9 \\
3   4 \\
};
\addlegendentry{$\epsilon=\expnumber{8}{-5}$}

\addplot 
  table[row sep=crcr]{%
0	16560 \\
1	1043 \\
2	81 \\
3	6 \\
};
\addlegendentry{$\epsilon=\expnumber{5}{-5}$}
\legend{}

\nextgroupplot[
title = \tikztitle{No. of samples, $p$-MLMC},
xlabel={level},
ymode=log,
ymin=1,
ymax=50000,
ylabel={no. samples},
xtick={0, 1, 2, 3},
xmin=-0.5,
xmax = 3.5,
yminorticks=true,
axis background/.style={fill=white},
legend style={at={(1.95,-0.55)}, anchor=south east, legend cell align=left, align=left, draw=white!15!black, font=\scriptsize},
legend columns=3,
cycle list name=black white,
]

\addplot 
  table[row sep=crcr]{%
0	12 \\
1	3 \\
2   3 \\
};
\addlegendentry{$\epsilon=\expnumber{1}{-3}$}

\addplot 
  table[row sep=crcr]{%
0	27 \\
1	3 \\
2	3 \\
};
\addlegendentry{$\epsilon=\expnumber{8}{-4}$}

\addplot 
  table[row sep=crcr]{%
0	103 \\
1	6 \\
2	3 \\
};
\addlegendentry{$\epsilon=\expnumber{5}{-4}$}

\addplot 
  table[row sep=crcr]{%
0	1983 \\
1	8 \\
2	3 \\
};
\addlegendentry{$\epsilon=\expnumber{1}{-4}$}

\addplot 
  table[row sep=crcr]{%
0	3041 \\
1	10 \\
2	5 \\
};
\addlegendentry{$\epsilon=\expnumber{8}{-5}$}

\addplot 
  table[row sep=crcr]{%
0	7694 \\
1	23 \\
2	11 \\
3	5 \\
};
\addlegendentry{$\epsilon=\expnumber{5}{-5}$}
\legend{}

\nextgroupplot[
title = \tikztitle{No. of samples, $hp$-MLMC},
xlabel={level},
ymode=log,
ymin=1,
ymax=50000,
ylabel={no. samples},
xtick={0, 1, 2, 3},
xmin=-0.5,
xmax = 3.5,
yminorticks=true,
xtick={0, 1, 2, 3},
xmin=-0.5,
xmax = 3.5,
axis background/.style={fill=white},
legend style={at={(0.95,-0.95)}, anchor=south east, legend cell align=left, align=left, draw=white!15!black, font=\scriptsize},
legend columns=2,
cycle list name=black white,
]

\addplot 
  table[row sep=crcr]{%
0	26 \\
1	3 \\
2   3 \\
};
\addlegendentry{$\epsilon=\expnumber{1}{-3}$}

\addplot 
  table[row sep=crcr]{%
0	49 \\
1	3 \\
2	3 \\
};
\addlegendentry{$\epsilon=\expnumber{8}{-4}$}

\addplot 
  table[row sep=crcr]{%
0	113 \\
1	8 \\
2	3 \\
};
\addlegendentry{$\epsilon=\expnumber{5}{-4}$}

\addplot 
  table[row sep=crcr]{%
0	3549 \\
1	666 \\
2	8 \\
3	3 \\
};
\addlegendentry{$\epsilon=\expnumber{1}{-4}$}

\addplot 
  table[row sep=crcr]{%
0	5723 \\
1	1032 \\
2	13 \\
3   4 \\
};
\addlegendentry{$\epsilon=\expnumber{8}{-5}$}

\addplot 
  table[row sep=crcr]{%
0	15345 \\
1	2897 \\
2	37 \\
3	9 \\
};
\addlegendentry{$\epsilon=\expnumber{5}{-5}$}

\nextgroupplot[
title = \tikztitle{average work},
scale only axis,
xmode=log,
xminorticks=true,
ymode=log,
xlabel={DOF},
ylabel={work},
ymax=1e8,
ymin=1e3,
yminorticks=true,
axis background/.style={fill=white},
legend pos = north west,
legend style={at={(0.95,-0.95)}, anchor=south east, legend cell align=left, align=left, draw=white!15!black, font=\scriptsize},
legend columns=2,
]
\addplot  [black,solid,mark=asterisk,] 
  table[row sep=crcr]{%
17856	15328.847971\\
27072	44466.8674976\\
61248	521318.530076\\
127168	2614584\\
};
\addlegendentry{$h$-MLMC}

\addplot [black,solid,mark=triangle] 
  table[row sep=crcr]{%
49675	86476.0395922\\
71532	253896.24\\
97363	1246024\\
127168	2874591.6\\
};
\addlegendentry{$p$-MLMC}

\addplot [black,solid,mark=diamond*] 
  table[row sep=crcr]{%
6975	8345.89794721\\
15228	22073.1939247\\
46893	271656\\
127168	2408907.73333\\
};
\addlegendentry{$hp$-MLMC}

\addplot[black,dashed] 
  table[row sep=crcr]{%
6975	48650.625\\
15228	231891.984\\
46893	2198953.449\\
127168	16171700.224\\
};
\addlegendentry{$\mathcal{O}(\text{DOF}^{2})$}

\end{groupplot}

\end{tikzpicture}%
  \settikzlabel{fig:cavity_samples_a} \settikzlabel{fig:cavity_samples_b} \settikzlabel{fig:cavity_samples_c}
 \caption{Computed number of samples and  average work on each level (Example \ref{exp:cavity}).}
 \label{fig:cavity_samples}
\end{figure}

Summarizing the numerical experiments
in Sections \ref{exp:smooth} and \ref{exp:cavity}
we have seen that $h$-, $p$- and $hp$-MLMC  yield
an optimal asymptotic runtime and can be used for efficient UQ of the compressible Navier--Stokes equations.
However, we also observed that the $hp$-MLMC method does not outperform  the $h$-MLMC 
method, as long as the latter can be applied. 
Therefore, in order to increase the efficiency of the $hp$-MLMC method future work 
should aim at   fully $hp$-adaptive algorithms as for example in \cite{Detommaso2019,Eigel2016,Kornhuber2018}.
\section{Conclusions}
In this article we have proposed the $hp$-MLMC method, a Discontinuous Galerkin based Multilevel Monte Carlo method,
where different levels consist of a hierarchy of uniformly refined spatial meshes in combination with a hierarchy of uniformly increasing DG polynomial degrees.
We generalized the complexity results from \cite{CliffeScheichl2011, AppeloeMOMC}
to arbitrarily $hp$-refined meshes. To account for the uncertainty of the optimal number of samples
on each level, we introduced a confidence interval, leading to a robust lower confidence bound 
for these quantities.
Our theoretical results are confirmed by numerical experiments for the two-dimensional compressible Navier-Stokes equations, including an extensive comparison between $h$-, $p$- and $hp$-MLMC methods. 
Finally, we applied our method to a challenging engineering 
problem from computational acoustics, demonstrating its capability to perform UQ for complex flow problems.
In order to improve the efficiency of the $hp$-MLMC method it should be advanced to an $hp$-\textrm{adaptive} method.
\vspace*{-0.2cm}
\section*{Acknowledgments}
The authors  thank the anonymous referees and the associate editor for their helpful suggestions to  improve
the article.
\bibliographystyle{siamplain}
\bibliography{mybib_revision_3}

\end{document}